\documentclass[reqno]{amsart}

\usepackage{amscd,amsmath,amssymb,amsfonts}
\usepackage{mathrsfs}
\usepackage[dvips]{color} 
\usepackage[all]{xy}

\numberwithin{equation}{section}

\theoremstyle{plain}
    \newtheorem{thm}{Theorem}[section]
    \newtheorem{lem}[thm]{Lemma}
    
    \newtheorem{prop}[thm]{Proposition}
    \newtheorem{cor}[thm]{Corollary}

\theoremstyle{definition}

    \newtheorem{rem}[thm]{Remark}

\begin{document}
\title{CM Periods, CM Regulators and Hypergeometric Functions, II}
\author{Masanori Asakura and Noriyuki Otsubo}
\address{Department of Mathematics, Hokkaido University, Sapporo, 060-0810 Japan}
\email{asakura@math.sci.hokudai.ac.jp}
\address{Department of Mathematics and Informatics, Chiba University, Chiba, 263-8522 Japan}
\email{otsubo@math.s.chiba-u.ac.jp}
\date{March 14, 2016}

\begin{abstract}
We study periods and regulators of a certain class of fibrations of varieties whose relative $H^1$ has multiplication by a number field. 
Both are written in terms of values of hypergeometric functions ${}_3F_2$ and the former reduces to values of the $\Gamma$-function, which provide examples of the conjecture of Gross-Deligne. 
\end{abstract}

\maketitle

\def\E{{\ol{D}^{(l)}_{ss}}}  

\def\can{\mathrm{can}}
\def\cano{\mathrm{canonical}}
\def\ch{{\mathrm{ch}}}
\def\Coker{\mathrm{Coker}}
\def\crys{\mathrm{crys}}
\def\dlog{d{\mathrm{log}}}
\def\dR{{\mathrm{d\hspace{-0.2pt}R}}}            % de Rham
\def\et{{\mathrm{\acute{e}t}}}  % etale
\def\Frac{{\mathrm{Frac}}}
\def\phami{\phantom{-}}
\def\id{{\mathrm{id}}}              % identity
\def\Image{{\mathrm{Im}}}        % image
\def\Hom{{\mathrm{Hom}}}  
\def\Ext{{\mathrm{Ext}}}
\def\MHM{{\mathrm{MHM}}}  
\def\MdRH{{\mathrm{MdRH}}}  
  
\def\ker{{\mathrm{Ker}}}          % kernel
\def\Ker{{\mathrm{Ker}}}          % kernel
\def\mf{{\text{mapping fiber of}}}
\def\Pic{{\mathrm{Pic}}}
\def\CH{{\mathrm{CH}}}
\def\NS{{\mathrm{NS}}}
\def\NF{{\langle D\rangle}}
\def\End{{\mathrm{End}}}
\def\Aut{{\mathrm{Aut}}}
\def\pr{{\mathrm{pr}}}
\def\Proj{{\mathrm{Proj}}}
\def\ord{{\mathrm{ord}}}
\def\tr{{\mathrm{tr}}}
\def\reg{{\mathrm{reg}}}          %
\def\res{{\mathrm{res}}}          %
\def\Res{\mathrm{Res}}
\def\Spec{{\mathrm{Spec}}}     % spectrum
\def\syn{{\mathrm{syn}}}
\def\cont{{\mathrm{cont}}}
\def\zar{{\mathrm{zar}}}
\def\HS{{\mathrm{HS}}}
\def\MHS{{\mathrm{MHS}}}
\def\VHS{{\mathrm{VHS}}}
\def\VMHS{{\mathrm{VMHS}}}
\def\CM{{\mathrm{CM}}}
\def\prim{{\mathrm{prim}}}
\def\fib{{\mathrm{fib}}}
\def\bA{{\mathbb A}}
\def\bC{{\mathbb C}}
\def\C{{\mathbb C}}
\def\G{{\mathbb G}}
\def\bE{{\mathbb E}}
\def\bF{{\mathbb F}}
\def\F{{\mathbb F}}
\def\bG{{\mathbb G}}
\def\bH{{\mathbb H}}
\def\bJ{{\mathbb J}}
\def\bL{{\mathbb L}}
\def\cL{{\mathscr L}}
\def\bN{{\mathbb N}}
\def\bP{{\mathbb P}}
\def\P{{\mathbb P}}
\def\bQ{{\mathbb Q}}
\def\Q{{\mathbb Q}}
\def\bR{{\mathbb R}}
\def\R{{\mathbb R}}
\def\bZ{{\mathbb Z}}
\def\Z{{\mathbb Z}}
\def\cH{{\mathscr H}}
\def\cD{{\mathscr D}}
\def\cE{{\mathscr E}}
\def\cO{{\mathscr O}}
\def\O{{\mathscr O}}
\def\cR{{\mathscr R}}
\def\cS{{\mathscr S}}
\def\cX{{\mathscr X}}
\def\cM{{\mathscr M}}
\def\cV{{\mathscr X}}
%                                 Greece
%
\def\g{\varepsilon}
\def\vG{\varGamma}
\def\vg{\varGamma}

%
%                                 simple
%
%
%
\def\lra{\longrightarrow}
\def\lla{\longleftarrow}
\def\Lra{\Longrightarrow}
\def\hra{\hookrightarrow}
\def\lmt{\longmapsto}
\def\ot{\otimes}
\def\op{\oplus}
%                              decolation
\def\wt#1{\widetilde{#1}}
\def\wh#1{\widehat{#1}}
\def\spt{\sptilde}
\def\ol#1{\overline{#1}}
\def\ul#1{\underline{#1}}
\def\us#1#2{\underset{#1}{#2}}
\def\os#1#2{\overset{#1}{#2}}

\def\om{{\Omega}}
\def\eq{{e_{\Q(\zeta_l)}}}
\def\dhcm{\check{H}_\CM}
\def\dk{\langle D\rangle^k}
\def\dinfk{\langle D_\infty\rangle^k}
\def\dssk{\langle \E\rangle^k}

%Macros added by Otsubo
\def\a{\alpha}\def\b{\beta}
\def\F#1#2#3{F\left({#1 \atop #2};#3\right)}

\section{Introduction}

In a previous paper \cite{asakura-otsubo}, we proved the Gross-Deligne period conjecture for a particular fibration of curves over the projective line. We also proved a formula which expresses regulators in terms of hypergeometric functions ${}_3F_2$. 
The aim of this paper is to prove similar results for more general fibrations. 
Firstly, the dimension of the fiber is arbitrary, but we put assumptions on the relative $H^1$. It is assumed to have multiplication by a number field $K$ and the dimension over $K$ is two and the monodromy is restricted. Secondly, the number field $K$ need not be abelian over $\Q$. This is beyond the scope of the original conjecture of Gross-Deligne  \cite{gross}. 

Let $f\colon X \to \P^1$ be a fibration equipped with $K$-multiplication on $R^1f_*\Q$ and satisfying our hypotheses (see Sect. 2). 
For a positive integer $l$, let $X^{(l)}$ be a desingularization of the base change of $X$ by the map $\pi:\P^1 \to \P^1; t \mapsto t^l$. 
Then, our first objective is the de Rham-Hodge structure 
\[
H^{(l)}:=H^1(\P^1,j_*(\pi_*\Q\ot R^1f_*\Q)),\quad
j:\P^1\setminus\{0,1,\infty\}\hra \P^1
\]
with $K[\Aut(\pi)]$-multiplication.
One easily sees that $H^{(l)}$ is a subquotient of $H^2(X^{(l)})$ and the complementary
space can be written explicitly
(see \S \ref{coh-sect} for the detail)

Letting $e_i:K[\Aut(\pi)]\to K_i$ be a projection to a number field, one has
a de Rham-Hodge structure $e_iH^{(l)}$ with $K_i$-multiplication.
Under some assumption on the monodromy, one can show $\dim_{K_i}e_i H^{(l)}=1$ and
hence we can discuss the {\it period} of the eigen-component 
$(H^{(l)})^\chi$ for each $\chi:K_i\hra\C$:
\[
(H^{(l)}_\dR)^\chi \cong \mathrm{Period}[(H^{(l)})^\chi]
\cdot (H^{(l)}_B)^\chi.
\]
If $K$ is abelian over $\Q$, the Gross-Deligne conjecture \cite{gross} states that the period is a product of values of the gamma function at rational numbers which reflect the Hodge decomposition of $H$. If $H$ is associated with $H^1$ of an abelian variety with complex multiplication by an abelian field, this is due to 
Shimura \cite{shimura} and Anderson \cite{anderson}. 
The elliptic case is the well-known Lerch-Chowla-Selberg formula \cite{lerch}. 

Our first main result is to compute the periods of $H^{(l)}$ and
verify the Gross-Deligne conjecture partially.
We note that our motive is not necessarily related with an abelian variety and $K$ may be non-abelian.

Our second main result of this paper is a regulator formula (Theorem \ref{main-reg}).
Beilinson's regulator map is a vast generalization of the classical regulator of units, and conjecturally describes a special value of the $L$-function.  
Our result describes a part of the regulator map 
$$\reg \colon H^3_\mathscr{M}(X^{(l)},\Q(2)) \to H^3_\mathscr{D}(X^{(l)},\Q(2))$$
from the motivic cohomology to the Deligne cohomology
in terms of special values of hypergeometric functions ${}_3F_2$. 
Recall that the classical polylogarithms are special case of hypergeometric functions and 
the regulators of Fermat curves are also written in terms of ${}_3F_2$-values \cite{otsubo1}. 
In our previous work \cite{asakura-otsubo}, we gave such a formula for a fibration of curves and proved the non-vanishing of the regulator. 
This amounts to compute the connecting homomorphism 
induced from the localization sequence of mixed Hodge structures (MHS) 
$$\rho \colon H_1(D_{ss},\Q) \to \Ext_\MHS^1(\Q, H_B^{(l)}\ot \Q(2))$$
where $D_{ss}$ denotes the fibers over $\mu_l$. 
Unfortunately our regulator formula does not guarantee
the non-vanishing of $\rho$, though we expect it in general.
We note that in our previous paper \cite{asakura-otsubo}, the non-vanishing of
regulator map is verified
in the case of the {\it hypergeometric fibrations} by developing a new technique.

The precise statements of our main theorems (the period formula and the regulator formula) shall be given in \S 4 (Theorems \ref{main}, \ref{main-reg}). 
The main ingredients of our method are {\it Fuchs equations} of
the hypergeometric functions and the theory of {\it limiting mixed Hodge structures}. 
We apply the theory of  Fuchs equations
to compute the periods of certain rational
relative $1$-forms and describe them in terms of hypergeometric functions.
Moreover, we use the theory of limiting mixed Hodge structures by Schmid to
determine certain coefficients of hypergeometric functions,
and 
in the proof of the regulator formula, {\it Dixon's formula} on ${}_3F_2$ also 
plays an important role.

Concerning the period conjecture of Gross-Deligne,  a further progress
was recently made by J. Fres\'{a}n and the first author
by a completely different method.
It covers our period formula (Theorem \ref{main}).
However, it seems impossible 
to obtain the regulator formula by the same method. 
Indeed we use lots of computational results in our proof of the period formula 
to prove the regulator formula. 

\medskip
This work is supported by JSPS Grant-in-Aid for Scientific Research, 24540001 and 25400007.

\medskip

%\noindent{\bf Acknowledgement}

\medskip

\noindent{\bf Notations}
Throughout this paper, we fix an embedding $\ol{\Q}\hra \C$.
For an algebraic variety $X$ defined over $\ol{\Q}$, we denote
$X^{an}:=\Hom_{\Spec\ol{\Q}}(\Spec\C, X)$ the associated analytic space.
We often omit ``$an$\rq\rq{} as it is clear from the context what is meant.
For example we denote
$H_B^\bullet(X,\Q)$ and
$H^B_\bullet(X,\Q)$ the Betti (co)homology of $X^{an}$. 
The hypergeometric function ${}_pF_q$ is defined by
$${}_pF_q\left({\alpha_1,\dots ,\alpha_p\atop \beta_1,\dots, \beta_q};x\right)=\sum_{n=0}^\infty \frac{\prod_{i=1}^p(\alpha_i)_n}{\prod_{j=1}^q(\beta_j)_n} \frac{x^n}{n!}$$
where $(\alpha)_n=\prod_{i=1}^n (\alpha+i-1)$ is the Pochhammer symbol. 
Recall that ${}_pF_q$ converges at $x=1$ if and only if $\sum\b_j-\sum\a_i>0$. 
We write 
$$\Gamma\left({\alpha_1,\dots, \alpha_p \atop \beta_1,\dots, \beta_q}\right) = \frac{\prod_{i=1}^p \Gamma(\alpha_i)}{\prod_{j=1}^q \Gamma(\beta_j)}.$$
Then, $B(\alpha,\beta)=\Gamma\left({\alpha,\beta \atop \alpha+\beta}\right)$ is the beta function.

\section{De Rham-Hodge Structure with Multiplication and Periods}\label{dR-H.sect}
\subsection{De Rham-Hodge structure}
A {\it de Rham-Hodge structure} (over $\ol{\Q}$) is a datum
$H=(H_\dR, H_B, \iota, F^\bullet)$ consisting of 
a finite dimensional vector space over $\ol\Q$ (resp. $\Q$)
$H_\dR$ (resp. $H_B$), a comparison isomorphism
$\iota\colon \C \ot_{\ol\Q} H_\dR \cong \C\ot_\Q H_B$ and
a filtration $F^\bullet$ on $H_\dR$ which induces a $\Q$-Hodge structure on $H_B$ via $\iota$. 
A {\it mixed de Rham-Hodge structure} (over $\ol{\Q}$) is a datum
$H=(H_\dR, H_B, \iota, F^\bullet,W_{\dR,\bullet},W_{B,\bullet})$ with
increasing filtrations $W_{\dR,\bullet}\subset H_\dR$, $W_{B,\bullet}\subset H_B$
such that each graded piece $\mathrm{Gr}^W_jH$ is a de Rham-Hodge structure of weight $j$.

Let $K$ be a $\Q$-algebra.
We call a ring homomorphism $\rho:K\to \End(H)$
a {\it multiplication} by $K$ where $\End$ denotes the endomorphism ring
of the mixed de Rham-Hodge structure. 
For an embedding $\chi:K\hra\ol{\Q}$, we set the eigenspaces as
\begin{align*}
&H_B^\chi:=\{x\in \ol{\Q}\ot_\Q H_B \mid gx=\chi(g)x, \forall g\in K\},\\
&H_\dR^\chi:=\{x\in H_\dR \mid gx=\chi(g)x, \forall g\in K\}. 
\end{align*}
If $K$ is a semisimple, commutative and 
finite dimensional $\Q$-algebra,
then one has the eigen-decompositions 
$\ol\Q \ot_\Q H_B=\bigoplus_\chi H_B^\chi$, $H_\dR=\bigoplus_\chi H_\dR^\chi$. 
If $K$ is a number field and $\dim_KH_B=1$ ($\Leftrightarrow$
$\dim_\Q H_B=[K:\Q]$), then the multiplication $\rho$ is called {\it maximal}.
In this case $H$ cannot have mixed weights.
Since $H_B$ is one-dimensional over $K$, one has $\dim_{\ol\Q} H_B^\chi
=\dim_{\ol\Q} H_\dR^\chi=1$ and then
there is a unique integer $p_\chi$ such that $H_B^\chi$
belongs to the Hodge component $H^{p_\chi,q_\chi}$.
The formal sum $T=\sum p_\chi\chi$ is called the {\it Hodge type} of $H$.

\subsection{Periods of de Rham-Hodge structure}
Let $H$ be a de Rham-Hodge structure with maximal multiplication.
For $\chi:K\hra\ol{\Q}$, 
there is a nonzero complex number $\mathrm{Period}(H^\chi)$ such that
$$\iota(e_\dR) = \mathrm{Period}(H^\chi) e_B$$
where $e_\dR \in H_\dR^\chi$ (resp. $e_B \in H^\chi_B$) is a basis. 
We call it the {\it period} of the $\chi$-part $H^\chi$.
This is well-defined up to multiplication by $\ol{\Q}^\times$.

\subsection{Variation of de Rham-Hodge structure}
Let $U$ be a smooth variety over $\ol{\Q}$. One can define
a variation of de Rham-Hodge structure $\cH=(H_\dR,H_B,\iota,F^\bullet,\nabla)$ on $U$
in a natural way.
A ring homomorphism $\rho:K\to \End(\cH)$ is called a {\it (relative) multiplication}
where
$\End$ denotes the endomorphism ring of the variation of de Rham-Hodge structure.

\section{Fibration with Relative Multiplication}\label{settei-sect}
We work over the base field $\ol{\Q}$.
We mean by a {\it fibration} a surjective morphism
\[
f:X\lra C
\]
from a smooth projective variety $X$ onto a smooth projective curve $C$. 
We mean by a {\it (relative)
multiplication} on $R^if_*\Q$ 
a ring homomorphism $\rho:K\to \End(R^if_*\Q|_U)$ with 
$U\subset C$ a (sufficiently small) non-empty Zariski open set.
\subsection{Setting and notation}
We begin with a fibration
\[
f:X\lra\P^1
\]
equipped with a relative multiplication on $R^1f_*\Q$
by a number field $K$ 
which satisfies the following conditions.
Hereafter we fix a coordinate $t\in \bA^1\subset \P^1$.
\begin{enumerate}
\item[\bf(a)]
The rank of the multiplication is two, i.e. $\dim_KR^1f_*\Q=2$.
\item[\bf(b)]
$f$ is smooth over $\P^1\setminus\{0,1,\infty\}$.
\item[\bf(c)]
The local monodromy $T=T_1$ on $H^1_B(X_t,\Q)$ at $t=1$
is maximally unipotent, i.e.
the rank of $N:=\log(T)$ is $\frac{1}{2}\dim_\Q H^1_B(X_t,\Q)$.
\end{enumerate}
Let $l\geq 1$ be an integer.
We then consider the commutative diagram
\[
\xymatrix{
X^{(l)}\ar[rd]_{f^{(l)}}\ar[r]^i&
X^*_l\ar@{}[rd]|{\square}\ar[r]\ar[d]&X\ar[d]^f\\
&\P^1\ar[r]^\pi&\P^1
}
\]
where $\pi(t)=t^l$ and $i$ is a desingularization.
Put $G^{(l)}:=\Aut(\pi)$. Note that $G^{(l)}$ is 
naturally isomorphic to
the group of $l$th roots of unity $\mu_l\subset \ol{\Q}^\times$.
We write by $\tau_\zeta$ the automorphism corresponding to $\zeta\in\mu_l$, namely
$\tau_\zeta(t)=\zeta t$.
There is a canonical isomorphism 
\[K[G^{(l)}]\cong \prod_i K_i\] of $\Q$-algebras where
$K_i$ are field extensions of $K$.
Let $e_i\in K[G^{(l)}]$ be the idempotent element corresponding to
the projection $K[G^{(l)}]\to K_i$ (i.e. $e_i^2=e_i$ and $e_iK[G^{(l)}]=K_i$).

For $k\in (\Z/l\Z)^\times$
let $\varepsilon_k:\Q[G^{(l)}]\to\ol{\Q}\subset \C$ be a homomorphism of $\Q$-algbra
given by $\varepsilon_k(\tau_\zeta)=\zeta^k$.
There is a one-to-one correspondence
\[
\begin{matrix}
\Hom_{\Q\mbox{-alg}}(K[G^{(l)}],\ol{\Q})
&\longleftrightarrow&
(\Z/l\Z)^\times
\times \Hom_{\Q\mbox{-alg}}(K,\ol{\Q}).\\
\varepsilon_k\ot\chi&\longleftrightarrow&
(k,\chi)
\end{matrix}
\]
Put
\begin{equation}\label{index}
I_i:=\{\varepsilon_k\ot\chi\mid \varepsilon_k\ot\chi\mbox{ factors through } K_i\}.
\end{equation}

Let $\Delta_p^*$ denote the punctured disk at $p=0,1$ or $\infty$.
Let $\phi:\pi_1(\Delta_p^*)\to \mathrm{GL}(H^1_B(X_t,\Q))$ 
be the local monodromy representation. Since the monodromy action is commutative
with the multiplication by $K$, it induces a two-dimensional representation
$\phi^\chi:\pi_1(\Delta_p^*)\to \mathrm{GL}(H^1_B(X_t,\Q)^\chi)\cong
\mathrm{GL}(2,\ol{\Q})$ for each $\chi:K\hra\ol{\Q}$ by the condition {\bf (a)}.
Its semisimplification $(\phi^\chi)^{ss}$ induces two homomorphisms
$\pi_1(\Delta_p^*)\to \mu_\infty\subset\ol{\Q}^\times$.
Under the canonical identifications
$\pi_1(\Delta_p^*)\cong H_1(\Delta_p^*,\Z)\cong\Z(1)$, $\mu_\infty\cong\Q/\Z(1)$
and $\Hom(\Z(1),\Q/\Z(1))\cong \Q/\Z$
they give rise to two rational numbers modulo integers, which we write by
$\alpha_1^\chi$ and $\alpha_2^\chi$ for $p=0$
and by $\beta_1^\chi$ and $\beta_2^\chi$ for $p=\infty$.
In other words, $e^{2\pi i\alpha^\chi_1}$ and
$e^{2\pi i\alpha^\chi_2}$ are eigenvalues of $T_0$ (=the local monodromy at $t=0$ in counter-clockwise
direction), and
$e^{2\pi i\beta^\chi_1}$ and
$e^{2\pi i\beta^\chi_2}$ are eigenvalues of $T_\infty$.

\subsection{Motivic sheaf $\cM^{(l)}$}\label{m1}
Put
\[
\cM^{(l)}:=\pi_*\Q\ot R^1f_*\Q\cong \pi_*(R^1f^{(l)}_*\Q)
\]
a variation of de Rham-Hodge structure on $\P^1\setminus\{0,1,\infty\}$ 
equipped with a multiplication
by the ring $K[G^{(l)}]$.
The stalk of $\cM^{(l)}$ is a free $K[G^{(l)}]$-module of rank $2$. 
The eigenvalues of $T_0$ (resp. $T_\infty$) 
on the $\varepsilon_k\ot\chi$-part of $\cM^{(l)}$ are
\[
\exp2\pi i\left(\frac{k}{l}+\alpha^\chi_1\right),\quad
\exp2\pi i\left(\frac{k}{l}+\alpha^\chi_2\right),
\]
\[(\mbox{resp. }
\exp2\pi i\left(-\frac{k}{l}+\beta^\chi_1\right),\quad
\exp2\pi i\left(-\frac{k}{l}+\beta^\chi_2\right)).
\]
\begin{lem}\label{ind-lemma}
Let $K[G^{(l)}]\to K_i$ be a projection and $e_i$ the associated idempotent.
Let $I_i$ be as in \eqref{index}. Fix an arbitrary $\varepsilon_{k_0}\ot\chi_0\in I_i$.
Then for any $\varepsilon_{k}\ot\chi\in I_i$ there is some 
$s\in \hat{\Z}^\times$ such that
\[
s\left(\frac{k_0}{l}+\alpha^{\chi_0}_j\right)=\frac{k}{l}+\alpha^\chi_j,\quad
s\left(\frac{k_0}{l}+\beta^{\chi_0}_j\right)=\frac{k}{l}+\beta^\chi_j\quad
\mbox{ in }\Q/\Z
\]
where $\hat{\Z}^\times:=\varprojlim_n(\Z/n\Z)^\times\cong \Aut(\Q/\Z)$
(by changing the numbering $\alpha_j^\chi$ and $\beta^\chi_j$ suitably).
\end{lem}
\begin{proof}
The Galois group $\mathrm{Gal}(\ol\Q/\Q)$ acts on $\ol\Q\ot_\Q e_i\cM^{(l)}$
by $\sigma\ot \id$ for $\sigma\in\mathrm{Gal}(\ol\Q/\Q)$.
Then letting $(\cM^{(l)})^{\varepsilon_{k}\ot\chi}\subset \ol\Q\ot_\Q e_i\cM^{(l)}$ 
be the $\varepsilon_{k}\ot\chi$-part, one has
\[
\sigma[(\cM^{(l)})^{\varepsilon_{k}\ot\chi}]\subset (\cM^{(l)})^{\sigma\circ(\varepsilon_{k}\ot\chi)}
\]
by definition.
There is $\sigma\in \mathrm{Gal}(\ol\Q/\Q)$ such that
$\sigma\circ(\varepsilon_{k_0}\ot\chi_0)=\varepsilon_{k}\ot\chi$.
For monodromy $T$ on $e_i\cM^{(l)}$, one has
\[
\det(1-xT|(\cM^{(l)})^{\varepsilon_{k}\ot\chi})
=\det(1-xT|(\cM^{(l)})^{\sigma\circ(\varepsilon_{k_0}\ot\chi_0)})
=\sigma\det(1-xT|(\cM^{(l)})^{\varepsilon_{k_0}\ot\chi_0})
\]
and hence
\[
\exp2\pi i\left(\frac{k}{l}+\alpha^\chi_j\right)
=\sigma\exp2\pi i\left(\frac{k_0}{l}+\alpha^{\chi_0}_j\right)
=\exp2\pi i \cdot s\left(\frac{k_0}{l}+\alpha^{\chi_0}_j\right),
\]
\[
\exp2\pi i\left(\frac{k}{l}+\beta^\chi_j\right)
=\sigma\exp2\pi i\left(\frac{k_0}{l}+\beta^{\chi_0}_j\right)
=\exp2\pi i\cdot s\left(\frac{k_0}{l}+\beta^{\chi_0}_j\right)
\]
where $s:=\varepsilon_{\mathrm{cyc}}(\sigma)\in\hat{\Z}^\times$ ($\varepsilon_{\mathrm{cyc}}=$ the cyclotomic character).
\end{proof}
\subsection{Monodromy on $H^1(X_t)$}
Let the notation and the assumption be
as above.
Let $X_t$ denote the general fiber of $f^{(l)}$.
Put 
$$S:=\P^1\setminus\{0,1,\infty\}, \quad S^{(l)}:=\pi^{-1}(S), \quad U^{(l)}:=(f^{(l)})^{-1}(S^{(l)}).$$
By {\bf(a)}, $H^1(X_t,\Q)$ is equipped with
an action of $K$ which commutes with the action of $\pi_1(S^{(l)},t)$.
Let $N:H^1(X_t,\Q)\to H^1(X_t,\Q)$ be the log monodromy at $t=1$.
Put $g=[K:\Q]$. The condition {\bf (b)} and the fact $N^2=0$ implies 
\[
\Ker(N)=\Image(N)\cong \Q^{\op g},\quad
\Coker(N)\cong \Q^{\op g}
\]and hence
\begin{equation}\label{lem4}
\ker(N)=\Image(N)\cong K,\quad \Coker(N)\cong K
\end{equation}
as $K$-modules.
By the theory of limiting MHS (\cite{Schmid}, \cite{steenbrink}),
the limit $H^1_{\mathrm{lim}}(X_t,\Q)$ at $t=1$ 
is a mixed Tate Hodge structure
such that 
\begin{equation}\label{lem4-0}
\mathrm{Gr}^W_iH^1_{\mathrm{lim}}(X_t,\Q)=
\begin{cases}
\ker(N)\cong \Q^{\op g}&i=0\\
\Coker(N)\cong \Q(-1)^{\op g}&i=2\\
0&\mbox{otherwise}.
\end{cases}
\end{equation}
\begin{lem}\label{monod-lem3}
Let $\Delta_1^*\subset \P^1$ be the punctured neighborhood of $t=1$.
Then there are isomorphisms
$H^1(\Delta_1^*,\cM^{(l)})\cong 
\Q(-2)^{\op gl}$ as de Rham-Hodge structure and
$H^1(\Delta_1^*,\cM^{(l)})\cong
K[G^{(l)}]$ as $K[G^{(l)}]$-module.
\end{lem}
\begin{proof}
Let $T_1$ be the local monodromy on $\cM^{(l)}$ at $t=1$.
Then there is a natural isomorphism
\[
H^1(\Delta_1^*,\cM^{(l)})\cong\Coker[T_1-1:\cM^{(l)}\to\cM^{(l)}\ot\Q(-1)].
\]
Now the assertion is immediate from \eqref{lem4}, \eqref{lem4-0} and the fact that the action of $T_1$ 
on $\pi_*\Q$ is trivial.
\end{proof}

\begin{lem}\label{lem5}
The invariant part of $H^1(X_t,\Q)$ by $\pi_1(S^{(l)},t)$ is zero.
\end{lem}
\begin{proof}
Let $M=H^1(X_t,\Q)^{\pi_1(S^{(l)},t)}$ be the invariant part.
By \cite{HodgeII} (4.1.2), $M$ is a sub-Hodge structure of pure weight $1$.
It defines a constant VHS, and hence the limiting MHS 
 $M_{\mathrm{lim}}$ at $t=1$ is of pure weight $1$.
Hence
\[
M=M_{\mathrm{lim}}\subset
\mathrm{Gr}^W_1H^1_{\mathrm{lim}}(X_t,\Q)=\Ker(N)/\Image(N)
\]
and the last term vanishes by \eqref{lem4}. Hence $M=0$.
\end{proof}
\begin{cor}\label{monod-lem1}
$\vg(S,\cM^{(l)})=0$ and
$\vg(S,(\cM^{(l)})^*)=0$ where $(-)^*$ denotes the dual sheaf.
\end{cor}
\begin{proof}
Indeed $\vg(S,\cM^{(l)})=\vg(S,\pi_*R^1f^{(l)}_*\Q)=\vg(S^{(l)},R^1f^{(l)}_*\Q)$
is the invariant part of $H^1(X_t,\Q)$ by $\pi_1(S^{(l)},t)$.
By the Hard Lefschetz theorem, $R^1f^{(l)}_*\Q\cong
(R^1f^{(l)}_*\Q)^*$ and hence the latter also follows.
\end{proof}
\begin{lem}\label{monod-lem2}
$H^2(\P^1,j_*\cM^{(l)})=0$ where $j:S\hra \P^1$ is the open immersion.
Hence $E=\vg(\P^1,R^1j_*\cM^{(l)})$ (see \eqref{local1} for the definition of $E$).
\end{lem}
\begin{proof}
Let $i:\P^1\setminus S\hra \P^1$ be the closed immersion.
There is an exact sequence
\[
0\lra j_!\cM^{(l)}\lra j_*\cM^{(l)}\lra i_*i^*j_*\cM^{(l)}\lra 0 
\]
of constructible sheaves.
This yields $H^2(\P^1,j_!\cM^{(l)})=H^2(\P^1,j_*\cM^{(l)})$.
By the Verdier duality theorem
\[
H^2(\P^1,j_!\cM^{(l)})=H^0(\P^1,Rj_*(\cM^{(l)})^*)^*=\vg(S,(\cM^{(l)})^*)^*.
\]
Now the assertion follows from Corollary \ref{monod-lem1}.
\end{proof}
\begin{lem}\label{lem6}
$H^1_B(X^{(l)},\Q)=0$.
\end{lem}
\begin{proof}
By Lemma \ref{lem5} one has $\vg(S^{(l)},R^1f^{(l)}_*\Q)=0$.
Hence $H^1(U^{(l)},\Q)=H^1(S^{(l)},\Q)$. Since $S^{(l)}$ is a smooth rational curve,
the weight of $H^1(U^{(l)},\Q)$ is $2$.
Hence the map $H^1(X^{(l)},\Q)\hra H^1(U^{(l)},\Q)$ must be zero.
This means 
$H^1(X^{(l)},\Q)=0$.
\end{proof}

\begin{lem}\label{Ev-lem}
Let $\mathrm{Ev}\subset H_1(X_t,\Q)$ be the space of vanishing cycles 
at $t=1$, namely $\mathrm{Ev}=\ker(N)$.
Then $\Q[\pi_1(S^{(l)},t)](\mathrm{Ev})=H_1(X_t,\Q)$.
Moreover for $\chi:K\hra\C$, one has
$\C[\pi_1(S^{(l)},t)](\mathrm{Ev}^\chi)=H_1(X_t,\C)^\chi$, where we put
$\mathrm{Ev}^\chi=\Ker(N)\cap H_1(X_t,\C)^\chi$.
\end{lem}
\begin{proof}
The latter follows immediately from the former.
We show the former.
By Deligne's semisimplicity theorem \cite{HodgeII} (4.1.6), 
there is a complementary subspace
$V\subset H_1(X_t,\Q)$ of 
$\Q[\pi_1(S^{(l)},t)](\mathrm{Ev})$
which is stable under the action of  $\pi_1(S^{(l)},t)$.
Note $\Image(N)=\Ker(N)=\mathrm{Ev}$. Therefore the commutative diagram
\[
\xymatrix{
H_1(X_t)/\mathrm{Ev}\ar[r]^{\quad N}_{\quad \cong}&\mathrm{Ev}\\
V\ar@{^{(}->}[r]^{N\quad}\ar[u]^\cup&V\cap \mathrm{Ev}=0\ar[u]
}
\]
yields $V=0$.
\end{proof}
\begin{lem}\label{Ev-cor}
$H_1(X_t,\C)^\chi\cong\C^2$ is an irreducible
$\C[\pi_1(S^{(l)},t)]$-module.
\end{lem}
\begin{proof}
By \eqref{lem4},
$\mathrm{Ev}^\chi$ is one-dimensional over $\C$.
Since $N^2=0$, one sees that  $\mathrm{Ev}^\chi$ is the 
unique one-dimensional subspace of $H_1(X_t,\C)^\chi$
which is stable under the action of $N$. 
Since $\C[\pi_1(S^{(l)},t)](\mathrm{Ev}^\chi)=H_1(X_t,\C)^\chi$
by Lemma \ref{Ev-lem},
there is no non-trivial $\C[\pi_1(S^{(l)},t)]$-submodule.
\end{proof}
\begin{lem}\label{hodgefilt-lem}
Let $F^\bullet H^i(X_t,\C)^{\chi}=F^\bullet \cap H^i(X_t,\C)^{\chi}$ 
denote the $\chi$-part
of the Hodge filtration. Then
$\dim_\C F^1H^1(X_t,\C)^{\chi}=1$.
\end{lem}
\begin{proof}
It is equivalent to say that
the limiting Hodge filtration $F^1_\infty\subset H^1_{\mathrm{lim}}(X_t,\C)$ at $t=1$
satisfies $\dim_\C F^1_\infty H^1_{\mathrm{lim}}(X_t,\C)^\chi=1$.
However, since $H^1_{\mathrm{lim}}(X_t,\C)$ is a MHS of type $(0,0)$ and $(1,1)$,
one has
\begin{equation}\label{hodgefilt-lem-eq}
F^1_\infty H^1_{\mathrm{lim}}(X_t,\C)\cong
\mathrm{Gr}^W_2H^1_{\mathrm{lim}}(X_t,\C)
=H^1_{\mathrm{lim}}(X_t,\C)/\Ker(N)\cong \Coker(N).
\end{equation}
By \eqref{lem4}, each eigenspace for the multiplication by $K$ is one-dimensional.
\end{proof}
\section{Main Theorems}
Let
\[
\cM^{(l)}:=\pi_*\Q\ot R^1f_*\Q\cong \pi_*(R^1f^{(l)}_*\Q)
\]
be as in \S \ref{m1}.
Let $j:\P^1\setminus\{0,1,\infty\}\hra \P^1$ be the open immersion.
We then consider the cohomology groups
\begin{equation}\label{hm-def}
H^{(l)}:=H^1(\P^1,j_*\cM^{(l)}),\quad
M^{(l)}:=H^1(\P^1\setminus\{0,1,\infty\},\cM^{(l)})
\end{equation}
and 
\begin{equation}\label{e-def}
E:=\bigoplus_{p=0,1,\infty}E_p,\quad
E_p:=(R^1j_*\cM^{(l)})_p=\Coker[T_p-1:\cM^{(l)}\to\cM^{(l)}]
\end{equation}
where $T_p$ is the local monodromy at $p$.
They carry de Rham-Hodge structures
by the theory of Hodge modules of M. Saito (\cite{msaito1}, \cite{msaito2}).
They are also equipped 
with multiplication by $K[G^{(l)}]$.
Since $H^2(\P^1,j_*\cM)=0$ (Lemma \ref{monod-lem2}), we have an exact sequence
\begin{equation}\label{local1}
\xymatrix{
0\ar[r] &H^{(l)}\ar[r]& 
M^{(l)}\ar[r]& E\ar[r]&0
}
\end{equation}
of de Rham-Hodge structures.
The de Rham-Hodge structure $H^{(l)}$ has a Hodge decomposition of type
$(0,2)+(1,1)+(2,0)$ (\cite{zucker-ann}).
It is auto-dual, namely there is an isomorphism
\begin{equation}\label{auto-dual}
(H^{(l)})^*\cong H^{(l)}\ot\Q(2)
\end{equation}
which arises from the isomorphism 
${\mathbb D}(j_*\cM^{(l)}[1])\cong j_*\cM^{(l)}[1]\ot\Q(2)$
of Hodge modules.
However \eqref{auto-dual} is not compatible with the action of
$K[G^{(l)}]$.
There is a unique involution $K\to K$, $\alpha\mapsto{}^t\alpha$
such that the pairing $(\, ,\,)$ on $R^1f_*\Q\ot R^1f_*\Q$
satisfies $(\alpha x,y)=(x,{}^t\alpha y)$. 
Then the involution $K[G^{(l)}]\to K[G^{(l)}]$, $g=\sum\alpha \sigma
\mapsto{}^tg:=\sum{}^t\alpha\sigma^{-1}$ induces a compatible action on 
\eqref{auto-dual} in the sense that
the pairing $(\, ,\,)$ on $H^{(l)}\ot H^{(l)}$ satisfies $(gx,y)=(x,{}^tg y)$.
In particular \eqref{auto-dual} induces
\[
[(H^{(l)})^*]^{\varepsilon_k\ot\chi}=
[(H^{(l)})^{\varepsilon_k\ot\chi}]^*\cong [H^{(l)}]^{\varepsilon_{-k}\ot{}^t\chi}\ot\Q(2).
\]
where ${}^t\chi(\alpha):=\chi({}^t\alpha)$ for $\alpha\in K$.

Since the paring on $R^1f_*\Q\ot R^1f_*\Q$ is compatible with the monodromy,
one has $\alpha_j^{{}^t\chi}=-\alpha^\chi_j$ and $\beta_j^{{}^t\chi}=-\beta^\chi_j$.

\subsection{Period formula}
Let $K[(G^{(l)}]\to K_i$ be a projection and $e_i$ the associated idempotent.
Since $\vg(S,\cM^{(l)})=0$ (Lemma \ref{monod-lem1})
one has
\[\dim_{\Q}e_iM^{(l)}
=-\chi(S,e_i\cM^{(l)})
=-\chi_{\mathrm{top}}(S)\dim_\Q (e_i\cM^{(l)})=2\dim_\Q K_i.\]
This means $e_iM^{(l)}\cong K_i^{\op 2}$ as $K_i$-module.
It follows from Lemma \ref{monod-lem3} that one has $E_1\cong K[G^{(l)}]$ as $K[G^{(l)}]$-module and hence $e_iE_1\cong K_i$.
Hence 
\[
\dim_{K_i}e_iH^{(l)}=1\quad\Longleftrightarrow\quad
e_iE_0=e_iE_\infty=0.
\]
The dimension of $e_iE_0$ or $e_iE_\infty$ does depend on $i$.
If none of eigenvalues of $T_0$ and $T_\infty$ is $1$, or equivalently
none of rational numbers
\[
\frac{k}{l}+\alpha^\chi_1,\quad
\frac{k}{l}+\alpha^\chi_2,\quad
-\frac{k}{l}+\beta^\chi_1,\quad
-\frac{k}{l}+\beta^\chi_2
\]
belongs to $\Z$ for all $\varepsilon_k\ot\chi\in I_i$ (cf. \S \ref{m1}),
then $e_iE_0=e_iE_\infty=0$, and hence
$e_iH^{(l)}$ is a de Rham-Hodge structure with maximal multiplication
by $K_i$.
 
\medskip 

Our first theorem is on the periods of $e_iH^{(l)}$:
\begin{thm}[Period formula]\label{main}
Let $K[G^{(l)}]\to K_i$ be a projection and $e_i$ the associated idempotent as in \S \ref{settei-sect}.
Suppose that none of rational numbers
\begin{equation}\label{main-cond}
\frac{k}{l}+\alpha^\chi_1,\quad
\frac{k}{l}+\alpha^\chi_2,\quad
-\frac{k}{l}+\beta^\chi_1,\quad
-\frac{k}{l}+\beta^\chi_2
\end{equation}
belongs to $\Z$ for some $\varepsilon_k\ot\chi\in I_i$ (and hence for all
$\varepsilon_k\ot\chi\in I_i$ by Lemma \ref{ind-lemma}).
Then the periods of
$e_iH^{(l)}$ are given as follows.
\[
\mathrm{Period}((e_iH^{(l)})^{\varepsilon_k\ot\chi})\sim_{\ol{\Q}^\times}
2\pi i\,
\Gamma\left({k/l+\alpha^\chi_1,
k/l+\alpha^\chi_2,\atop  k/l-\beta^\chi_1,k/l-\beta^\chi_2}\right).
\]
\end{thm}
We note that the auto-duality \eqref{auto-dual} yields
\begin{equation}\label{auto-dual-per}
(2\pi i)^2\mathrm{Period}([e_iH^{(l)}]^{\varepsilon_k\ot\chi})^{-1}
\sim_{\ol{\Q}^\times}\mathrm{Period}([e_iH^{(l)}]^{\varepsilon_{-k}\ot{}^t\chi}).
\end{equation}
One can directly check it on noting
$\alpha_j^{{}^t\chi}=-\alpha^\chi_j$, $\beta_j^{{}^t\chi}=-\beta^\chi_j$
and $\Gamma(x)\Gamma(1-x)=\pi/\sin(\pi x)$.

%where $\ol\chi$ is the complex conjugation of the embedding $\chi:K\hra\C$

\begin{rem}
J. Fres\'{a}n and the first author recently verified the period conjecture of Gross-Deligne
for the determinant of cohomology groups, and it includes our motive $H^{(l)}$.
In particular, it is proven that the Hodge type $p_{\varepsilon_k\ot \chi}$ of $H^{(l)}$
is given as follows,
\[
p_{\varepsilon_k\ot \chi}=1+\left\{\frac{k}{l}+\a^\chi_1\right\}
+\left\{\frac{k}{l}+\a_2^\chi\right\}
-\left\{\frac{k}{l}-\b^\chi_1\right\}-\left\{\frac{k}{l}-\b_2^\chi\right\}
\]
where $\{x\}:=x-\lfloor x\rfloor$ is the fractional part.
\end{rem}
\subsection{Regulator formula}
Our second main result is on the extension data of the exact sequence \eqref{local1}.

\medskip

Again let $K[G^{(l)}]\to K_i$ and $e_i$ satisfy the assumption in Theorem \ref{main}.
Recall from Lemma \ref{monod-lem3} that $E_1$ is isomorphic to
a direct sum of $\Q(-2)$ as a de Rham-Hodge structure.
The exact sequence
\eqref{local1} gives rise to the connecting homomorphism
\begin{equation}\label{connecting-def}
\rho:E_1(2)\lra \Ext^1_\MdRH(\Q,H^{(l)}(2)),\quad (V(j):=V\ot\Q(j))
\end{equation}
where $\Ext_\MdRH$ denotes the Yoneda extension group in the category of 
mixed de Rham-Hodge structures.
For $\varepsilon_k\ot\chi\in I_i$,
let $\delta_{\varepsilon_k\ot\chi}:=\dim_{\ol\Q}[e_iF^2H_\dR^{(l)}]^{\varepsilon_k\ot\chi}
=0$ or $1$.
We define $\rho^{\varepsilon_k\ot\chi}$ to be the composition of $\rho$ and 
\begin{align*}
\Ext^1_\MdRH(\Q,H^{(l)}(2))&\to
\Ext^1_\MdRH(\Q,e_iH^{(l)}(2))\quad \mbox{(projection)}\\
&\cong 
\Coker[
e_iH_B^{(l)}(2)\to \C\ot_{\ol\Q}(e_iH_\dR^{(l)}/F^2)]\\
&\to
\Coker[
(e_iH_B^{(l)}(2))^{\varepsilon_k\ot\chi}\to
\C\ot_{\ol\Q}(e_iH_\dR^{(l)}/F^2)^{\varepsilon_k\ot\chi}]\\
&\cong \C/[\ol{\Q}\delta_{\varepsilon_k\ot\chi}+\ol{\Q}\cdot(2\pi i)^2\mathrm{Period}([e_iH^{(l)}]^{\varepsilon_k\ot\chi})^{-1}]
\\
&\cong \C/
[\ol{\Q}\delta_{\varepsilon_k\ot\chi}+
\ol{\Q}\cdot\mathrm{Period}([e_iH^{(l)}]^{\varepsilon_{-k}\ot{}^t\chi}]
\end{align*}
%\footnote{$H_2(X^{(l)},\ol\Q)^\chi$ should be $(H_2(X^{(l)},\ol\Q)^\ph)^\chi$?}
%\footnote{maybe $H^\chi \subset F^1$ is assumed, otherwise the map is not defined.}
where the second isomorphism is given with respect to a $\ol{\Q}$-basis of
$(e_iH_\dR^{(l)})^{\varepsilon_k\ot\chi}\cong\ol{\Q}$
and the last isomorphism follows from \eqref{auto-dual-per}.
Obviously $\rho^{\varepsilon_k\ot\chi}$ factors through $e_iE_1(2)=e_iE(2)$ or
the $\varepsilon_k\ot\chi$-part $[e_iE_1(2)]^{\varepsilon_k\ot\chi}\cong\ol{\Q}$.
\begin{thm}[Regulator formula]\label{main-reg}
Let the notation and the assumption be
as in Theorem \ref{main}.
There is a complex number $c=c_{f,\chi}\in 
\ol{\Q}+2\pi i\ol{\Q}+
\sum_{a\in\ol{\Q}^\times}\ol{\Q}\log(a)$ depending only on $f:X\to\P^1$
and $\chi$ such that the following holds.
Let $\varepsilon_k\ot\chi\in I_i$ and $x\in e_iE(2)=e_iE_1(2)$.
Then $\rho^{\varepsilon_{-k}\ot{}^t\chi}(x)$ is
a $\ol\Q$-linear combination of 
\begin{equation}\label{main-reg-term0}
1,\quad
c \cdot\Gamma\left({\alpha^\chi_1+k/l,
\alpha^\chi_2+k/l \atop k/l-\beta^\chi_1, k/l-\beta^\chi_2}\right), 
\end{equation}
and
\begin{equation}\label{main-regterm}
B(\alpha^\chi_1+\beta^\chi_1,\alpha^\chi_1+\beta^\chi_2)~
{}_3F_2\left(\begin{matrix}\alpha^\chi_1+\beta^\chi_1,\alpha^\chi_1+\beta^\chi_2,
\alpha^\chi_1+k/l\\
2\alpha^\chi_1+\beta^\chi_1+\beta^\chi_2,\alpha^\chi_1+k/l+1\end{matrix};1
\right).
\end{equation}
In addition the coefficient of \eqref{main-regterm} is non-zero unless $x=0$.
\end{thm}
There is an alternative description of the $\varepsilon_k\ot \chi$-part 
of $\rho$.
Let $\mathrm{Filt}_{\ol\Q}$
be the category of finite dimensional $\ol\Q$-modules equipped with finite decreasing filtration.
Let $\mathrm{Vec}_{\ol\Q}$ (resp. $\mathrm{Vec}_\C$) be the
category of finite dimensional $\ol\Q$-modules (resp. $\C$-modules).
Let $\mathrm{MF}=\mathrm{MF}_{\dR,B}:=\mathrm{Vec}_{\ol\Q}
\times_{\mathrm{Vec}_\C}\mathrm{Filt}_{\ol\Q}$ whose objects consist of
$M=(M_\dR,M_B,F^\bullet,\iota)$ where $M_B\in \mathrm{Vec}_{\ol\Q}$
and $(M_\dR,F^\bullet)\in \mathrm{Filt}_{\ol\Q}$ and $\iota:
\C\ot M_\dR\cong \C\ot M_B$ is a comparison isomorphism.
This is not abelian but is an {\it exact category} in which all morphisms have kernel and cokernel. Therefore one can discuss the Yoneda extension
groups $\Ext^j_{\mathrm{MF}}(M',M)$ (see also \cite{BBDG} 1.1 for
the derived category of $\mathrm{MF}$).
In a similar way to \cite{beilinson} one can show that there is a natural isomorphism
\[%\begin{equation}\label{beilinson}
\Ext^1_{\mathrm{MF}}(\ol\Q,M)\cong (\C\otimes M_\dR)/(F^0M_\dR+\iota^{-1}(M_B))
\]%\end{equation}
where $\ol\Q$ denotes the trivial one-dimensional object.
There is an exact functor from 
the category of mixed de Rham-Hodge structures with multiplication by $K[G^{(l)}]$
to $\mathrm{MF}$ given by
$M\mapsto M^{\varepsilon_k\ot\chi}$.
This induces
\[
\Ext^1(\ol\Q,H)\lra
\Ext^1_{\mathrm{MF}}(\ol\Q,H^{\varepsilon_k\ot\chi})
\cong (\C\otimes H^{\varepsilon_k\ot\chi}_\dR)
/(F^0H^{\varepsilon_k\ot\chi}_\dR+\iota^{-1}(H_B^{\varepsilon_k\ot\chi}))
\]
where the first $\Ext$ is the Yoneda extension group in 
the category of mixed de Rham-Hodge structures with multiplication by $K[G^{(l)}]$.
The exact sequence
\eqref{local1} gives rise to the connecting homomorphism $E_1(2)\to
\Ext^1(\ol\Q,H)$, and then $\rho^{\varepsilon_k\ot\chi}$ is the composition of this
with the above.
\subsection{Motivic interpretation of the mixed Hodge structure $M^{(l)}$}
The connecting homomorphism \eqref{connecting-def} describes
{\it Beilinson's regulator map} on a motivic cohomology group.

Put $D^{(l)}_0:=(f^{(l)})^{-1}(0)$,
$D^{(l)}_\infty:=(f^{(l)})^{-1}(\infty)$,
$D^{(l)}_i:=(f^{(l)})^{-1}(\zeta_l^i)$ ($1\leq i\leq l$)
and $D^{(l)}:=D_0^{(l)}+D_\infty^{(l)}+\sum D^{(l)}_i$.
Put $U^{(l)}=X^{(l)}\setminus D^{(l)}$.
The exact sequence \eqref{local1} sits into the following commutative diagram
\begin{equation}\label{local1-com}
\xymatrix{
&0\ar[d]&0\ar[d]\\
0\ar[r]& H^{(l)}\ar[r]\ar[d]^\iota& M^{(l)}\ar[r]\ar[d]& E\ar[r]\ar[d]^\cap&0\\
0\ar[r]& H^2(X^{(l)})/\langle D^{(l)}\rangle\ar[r]\ar[d]& H^2(U^{(l)})\ar[r]\ar[d]
& H^3_{D^{(l)}}(X^{(l)})\\
& H^2(X^{(l)}_t)\ar@{=}[r]& H^2(X_t^{(l)})
}
\end{equation}
where $X^{(l)}_t=(f^{(l)})^{-1}(t)$ is the general fiber and
$\langle D^{(l)}\rangle$ is the subgroup generated by the cycle classes of the irreducible components
of $D^{(l)}$.
\begin{prop}
%\footnote{Here, $X^{(l)}$ and $D_{ss}$, etc.  seem base extension to $\C$. Otherwise, %$\reg_{D_{ss}}$ may not be surjective}
Put $D^{(l)}_{ss}:=D^{(l)}-(D^{(l)}_0+D^{(l)}_\infty)$.
Then the diagram
\[
\xymatrix{
H_{\cM,D_{ss}}^3(X^{(l)},\Q(2))\ar[r]\ar[d]_{\reg_{D_{ss}}}
&H_{\cM}^3(X^{(l)},\Q(2))
\ar[d]^\reg\\
E_1(2)\ar[r]^{\iota\circ\rho\qquad \qquad}&
\Ext^1_\MHS(\Q,(H^2(X^{(l)})/\langle D^{(l)})(2)\rangle)
}\]
is commutative up to sign.
Moreover the map $\reg_{D_{ss}}$ is surjective.
\end{prop}
\begin{proof}
See \cite{sato} 11.2 for the commutativity.
We can see the surjectivity of $\reg_{D_{ss}}$ in the following way.
Let $D_{ss}^\circ\subset D_{ss}$ be the regular locus and put 
$Z:=D_{ss}\setminus D_{ss}^\circ$.
There is a commutative diagram
\[
\xymatrix{
0\ar[r]&H^3_{\cM,D_{ss}}(X^{(l)},\Q(2))\ar[d]_{\reg_{D_{ss}}}\ar[r]&
\O(D_{ss}^\circ)^\times\ot\Q\ar[r]\ar[d]^{\mathrm{dlog}}
&\Q Z\ar@{^{(}->}[d]\\
0\ar[r]&H^3_{D_{ss}}(X^{(l)},\Q(2))\ar[r]&
H^3_{D_{ss}^\circ}(X^{(l)}\setminus Z,\Q(2))\ar[r]
&H^4_Z(X^{(l)},\Q(2))\\
&&H^1(D_{ss}^\circ,\Q(1))\ar[u]_\cong
}
\]
with exact rows.  As is easily shown, $\mathrm{dlog}$ is surjective
onto $H^1(D_{ss}^\circ,\Q(1))\cap H^{0,0}$.
Hence so is $\reg_{D_{ss}}$ onto 
$H^3_{D_{ss}}(X^{(l)},\Q(2))\cap H^{0,0}=E_1(2)$.
\end{proof}

%%%%%%%%%%%%%%%%%

\section{Key Lemmas}
In this section we prove three lemmas which play key roles in the proof of
Theorems \ref{main} and \ref{main-reg}. 
Let the notation be as in \S \ref{settei-sect}.
We fix an arbitrary embedding $\chi:K\hra \C$ throughout this section, and 
simply write $\alpha_j=\alpha^\chi_j$ and $\beta_j=\beta^\chi_j$.

\subsection{Key Lemma 1}
Let $(\cH:=R^1f_*\Omega^\bullet_{U/S},\nabla)$ be the connection
on $S=\P^1\setminus\{0,1,\infty\}$ with regular singularities at $t=0,1,\infty$.
The number field $K$ acts on $(\cH,\nabla)$.
Let
\[
(\cH^\chi,\nabla)\subset (\cH,\nabla)
\]
be the $\chi$-part, a connection of rank 2.
The Hodge filtration $F^1\cH^\chi$ is a subbundle of rank 1
(Lemma \ref{hodgefilt-lem}). 

Fix a relative $1$-form 
$\omega\ne0\in \vg(S,F^1\cH^\chi)\subset\vg(U,\Omega^1_{U/S})$ which is defined over $\ol{\Q}$.
Let $N:H_1(X_t,\Q)\to H_1(X_t,\Q)$ be the log monodromy at $t=1$.
The eigenvalues
of the local monodromy on $H_1(X_t,\C)^\chi\cong \C^2$ at $t=0$
(resp. $t=\infty$) is written as $\{e^{2\pi i\alpha_1},e^{2\pi i\alpha_2}\}$ 
(resp. $\{e^{2\pi i\beta_1},e^{2\pi i\beta_2}\}$).

\begin{lem}\label{lem-basis}
There exists a basis $\{\delta_t, \gamma_t\}$ of $H_1(X_t,\ol{\Q})^\chi$ such that
\begin{align}
&
(T_1(\delta_t),T_1(\gamma_t))=
(\delta_t,\gamma_t)\begin{pmatrix}1&1\\0&1\end{pmatrix}, \label{T_1}
\\&
(T_0(\delta_t),T_0(\gamma_t))=
(\delta_t,\gamma_t)\begin{pmatrix}e^{2\pi i\alpha_2}&0\\ \varepsilon&
e^{2\pi i\alpha_1}\end{pmatrix}, \label{T_0}
\end{align}
for some $\varepsilon\in\ol{\Q}$, $\varepsilon\ne 0$.
We have $\alpha_1+\alpha_2+\beta_1+\beta_2 \in \Z$ and $\alpha_i+\beta_j \not\in \Z$ for any $i$, $j$. 
\end{lem}

\begin{proof}
Since $T_1$ is a non-trivial unipotent monodromy
on $H_1(X_t,\ol{\Q})^\chi$ by \eqref{lem4},
there is a unique eigenvector $\delta_t$ such that $T_1(\delta_t)=\delta_t$.
Let $\gamma_t$ be any cycle which is linearly independent from $\delta_t$.
Then $T_1(\gamma_t)-\gamma_t=c\delta_t$
for some $c\in\ol{\Q}^\times$.
By replacing $\delta_t$ with $c^{-1}\delta_t$, we obtain \eqref{T_1}. 
Secondly, if $\delta_t$ is an eigenvector for $T_0$, then the subspace $\ol{\Q}\cdot\delta_t$
is stable under the action of $\pi_1(S,t)$. This contradicts Lemma \ref{Ev-cor}.
Therefore an eigenvector of $T_0$ must be $\gamma_t+c\delta_t$ for some $c\in\ol{\Q}^\times$.
Replacing $\gamma_t$ with $\gamma_t-c\delta_t$, we have \eqref{T_0} for some $\varepsilon \in \ol\Q$. 
Again by Lemma \ref{Ev-cor}, we have $\varepsilon \ne 0$. 
Hence the first assertion is proved. 
Since 
$\mathrm{Tr}(T_1T_0)=\mathrm{Tr}(T_\infty^{-1})=e^{-2\pi i\beta_1}+e^{-2\pi i\beta_2}$, 
we have
$$
-e^{2\pi i\alpha_1}-e^{2\pi i\alpha_2}+e^{-2\pi i\beta_1}+e^{-2\pi i\beta_2}=\varepsilon\ne 0.
$$
On the other hand, since $T_\infty T_1T_0=I$, we have
$$e^{2\pi i\alpha_1}e^{2\pi i\alpha_2}e^{2\pi i\beta_1}e^{2\pi i\beta_2}=1.$$ 
These imply the second assertion. 
\end{proof}

From now on, we assume that 
$\alpha_1+\alpha_2+\beta_1+\beta_2=1$. 
For $\delta_t$, $\gamma_t$ as above, we put
\begin{equation*}\label{key1}
f_1(t):=\int_{\delta_t}\omega,\quad
f_2(t):=\int_{\gamma_t}\omega, 
\end{equation*}
which are multi-valued analytic functions on $S$.

\begin{lem}[Key Lemma 1]\label{key-lem}
Put
$$
F_1(t):=t^{\alpha_1}{}_2F_1\left({\alpha_1+\beta_1,\alpha_1+\beta_2 \atop 1};1-t\right),\quad
F_2(t):=t^{\alpha_1}{}_2F_1\left({\alpha_1+\beta_1,\alpha_1+\beta_2 \atop 1+\alpha_1-\alpha_2};t\right). 
$$
Then, there is a differential operator $\theta = q_0 + q_1 \frac{d}{dt}$ with
$q_i(t)\in \ol{\Q}(t)$ and constants $\lambda_i\in\C$ such that
\[
f_1=\lambda_0 \theta F_1, \quad
f_2=\lambda_1 \theta F_1+\lambda_2 \theta F_2.
\]
Moreover, $\lambda_0\lambda_2\ne 0$. 
\end{lem}

\begin{proof}
As is well-known, $f_i(t)$ are linearly independent solutions of the Fuchs equation
(=ordinary differential equation with regular singularities) arising from 
$(\mathscr{H}^\chi, \nabla)$.
Therefore it is completely determined by the monodromy of $f_1$ and $f_2$. 
Then as is well-known, its monodromy is isomorphic to that of
$H_1(X_t)^\chi$.
By the above lemma, it is expressed by the Riemann scheme
$$\left\{
\begin{matrix}
t=0&t=1&t=\infty\\
\alpha_1&0&\beta_1\\
\alpha_2&0&\beta_2
\end{matrix}
\right\}.$$
This coincides with that of the Gauss hypergeometric equation whose solutions are $F_1$ and $F_2$. 
The fundamental theorem of Fuchs equations 
(Riemann-Hilbert correspondence) yields the existence of a differential operator
$\theta=q_0(t)+q_1(t)d/dt$ ($\exists q_i(t)\in \C(t)$)
such that
\[
\langle f_1,f_2\rangle_\C
=\langle\theta F_1,\theta F_2\rangle_\C.
\]
Here $\theta$ gives an equivalence of the Fuchs equation of $f_i$ and that of $F_i$. 
Since both equations are defined over $\ol{\Q}$, $q_i(t)$ are defined over $\ol{\Q}$.
Finally, $f_1$ is characterized as an eigenfunction for $T_1$ and so is $F_1$.
Therefore $\langle f_1\rangle_\C=\langle \theta F_1\rangle_\C$. 
Since $\{f_1,f_2\}$ are linearly independent, $\lambda_0\lambda_2\ne0$ follows.
\end{proof}

\subsection{Key Lemma 2}
\begin{lem}[Key Lemma 2]\label{key2}
Let the notation be as in Lemma \ref{key-lem} (Key Lemma 1).
Then $\lambda_0\in 2\pi i\ol{\Q}^\times$.
\end{lem}
\begin{proof}
Let $\delta_{i,t}\in H_1(X_t,\Q)\cap\Ker(N)$ $(1\leq i\leq g$)
be a basis. 
Then $f_1(t)$ is a linear combination of
\[
g_i(t)=\int_{\delta_{i,t}}\omega
\]
over $\ol{\Q}$.
Let $J_t$ be the Albanese variety of $X_t$.
Since $J_t$ degenerates totally at $t=1$, 
there is an algebraic uniformization
\[
u:(\G_{m,r})^g=\Spec\ol{\Q}[\varepsilon]/(\varepsilon^r)[u_i,u_i^{-1}]_{1\leq i\leq g}
\lra J_t,\quad 
\varepsilon:=t-1
\]
for $r\geq 1$.
Thinking $\omega$ of a 1-form on $J_t$, let
\[
u^*(\omega)=\sum_{i=1}^g h_i(t)\frac{du_i}{u_i},\quad
h_i(t)\in \ol{\Q}((t-1)).
\]
Then
\[
g_i(t)=\sum_{j=1}^g h_j(t)\int_{\delta_{i,t}}\frac{du_j}{u_j}=
2\pi i\sum_{j=1}^g r_{ij}h_j(t),\quad 
r_{ij}:=\frac{1}{2\pi i}\int_{\delta_{i,t}}\frac{du_j}{u_j}\in\Q.
\]
Therefore we have
\[
\lambda_0\theta F_1(t)=f_1(t)=\sum_{i=1}^g c_i g_i(t)\in 2\pi i\ol{\Q}((t-1)),\quad \exists
c_i\in\ol{\Q}.
\]
Since $\theta$ is a differential operator with coefficients in $\ol{\Q}$ and
$F_1$ has a Taylor expansion at $t=1$ with coefficients in $\Q$,
the assertion follows.
\end{proof}
\subsection{Key Lemma 3}
\begin{lem}[Key Lemma 3]\label{key3}
Let the notation be as in Lemma \ref{key-lem} (Key Lemma 1).
Then 
\[
\lambda_2=-\frac{\lambda_0}{2\pi i}B(\alpha_1+\beta_1,\alpha_1+\beta_2), \quad 
\lambda_1\in \ol{\Q}+2\pi i\ol{\Q}+\sum_{a\in\ol{\Q}^\times}\ol{\Q}\log(a).
\]
\end{lem}
Key lemma 3 is proven by looking at the asymptotic behavior of $f_2$ at $t=1$.
To do this, we first prepare the following lemma, which is proven by
the theory of limiting mixed Hodge structures due to Schmid \cite{Schmid}
and a theorem of Hoffman \cite{hoffman}.
\begin{lem}\label{schmid-lem}
Put
$$\wt{f}_2:=f_2-\frac{1}{2\pi i}\log(1-t)f_1. $$
Then we have
\begin{enumerate}
\item[(i)]
$f_1$ and $\wt{f}_2$ are meromorphic at $t=1$.
\item[(ii)]
$ \ord_{t=1}(f_1)\leq \ord_{t=1}(\wt{f}_2)$.
\item[(iii)]
$\lim_{t\to 1}2\pi i\wt{f}_2(t)/f_1(t)\in
2\pi i\ol{\Q}+\sum_{a\in\ol{\Q}^\times}\ol{\Q}\log(a)$.
\end{enumerate}
\end{lem}
\begin{proof}
Let $\delta^*_t,\gamma^*_t\in H^1(X_t,\C)^\chi$ be the dual basis.
Then
\[
\omega=\left(\int_{\delta_t}\omega\right)\delta_t^*
+\left(\int_{\gamma_t}\omega\right)\gamma_t^*
=f_1(t)\delta_t^*+f_2(t)\gamma_t^*\in H^1(X_t,\C)^\chi.
\]
%Let $N$ be the log monodromy at $t=1$. 
Note that $N(\delta^*_t)=\gamma_t^*$ and $N(\gamma_t^*)=0$.
The nilpotent orbit theorem of Schmid \cite{Schmid} yields that the subspace spanned by
\[
\exp\left(\frac{1}{2\pi i}\log(1-t)N\right)\omega=f_1(t)\delta_t^*+\wt{f}_2(t)\gamma_t^*
\in
H^1_{\mathrm{lim}}(X_t,\C)^\chi
\]
converges in the flag manifold as $t\to 1$.
Since $f_1=\lambda_0\theta F_1$ is meromorphic at $t=1$, 
so is $\wt{f}_2(t)$. This proves (i).

Let $k_1:=\ord_{t=1}(f_1)$ and
$k_2:=\ord_{t=1}(\wt{f}_2)$.
Suppose $k_1>k_2$.
Then the limiting Hodge filtration $F^1_\infty$ is spanned by
\[
\left(\lim_{t\to 1}f_1(t)/\wt{f}_2(t)\right)\delta_t^*+\gamma_t^*
=\gamma_t^*\in
H^1_{\mathrm{lim}}(X_t,\C)^\chi.
\]
Namely $F_\infty^1=\Ker(N)$. This is impossible by \eqref{hodgefilt-lem-eq}.
Hence we have $k_1\leq k_2$, finishing the proof of (ii).

Finally we show (iii).
Since $k_1\leq k_2$, the limiting Hodge filtration is spanned by 
\[
\delta_t^*
+(\lim_{t\to 1}\wt{f}_2(t)/f_1(t))\gamma_t^*
\in
H^1_{\mathrm{lim}}(X_t,\C)^\chi.
\]
The main theorem of \cite{hoffman} yields that
the extension data of
\[
0\lra \Q(1)^{\op g}\lra H^1_{\mathrm{lim}}(X_t,\Q(1))\lra \Q^{\op g}\lra 0
\]
are $\log(\ol{\Q}^\times)$. Therefore
$\lim_{t\to 1}2\pi i\wt{f}_2(t)/f_1(t)$ is a linear combination of $\log(\ol{\Q}^\times)$
over $\ol{\Q}$, as desired.
\end{proof}

\begin{proof}[Proof of Key Lemma 3]
%\noindent({\it Proof of Key Lemma 3}).
%Let $N$ be the log monodromy at $t=1$.
By Lemma \ref{lem-basis}, we have $Nf_2=f_1$.
By Lemma \ref{key-lem} (Key lemma 1),
$\lambda_2\theta(NF_2)=\lambda_0\theta F_1$
and hence
$\lambda_2NF_2=\lambda_0F_1$.
The asymptotic behaviour of Gauss\rq{} hypergeometric function is given 
as follows (see \cite{Bateman} 2.3.1, p.74):
\begin{equation}\label{key3-eq1}
{}_2F_1\left({a,b\atop a+b};t\right)
=B(a,b)^{-1}\sum_{n=0}^\infty\frac{(a)_n(b)_n}{n!^2}(k_n-\log(1-t))(1-t)^n
\end{equation}
where 
\[
k_n:=2\psi(n+1)-\psi(a+n)-\psi(b+n),\quad
\psi(z):=\Gamma'(z)/\Gamma(z).
\]
In particular, we have
\[
\lim_{t\to 1}NF_2=
-2\pi i/B(a,b)
\]
with $a=\a_1+\b_1$, $b=\a_1+\b_2$. 
Comparing with $F_1(1)=1$, we obtain the first assertion. 
Next, by \eqref{key3-eq1}, there are analytic functions $h_i$ at $t=1$ such that  
$$\theta F_2=B(a,b)^{-1}
(h_1+h_2\log(1-t)).$$
By Lemma \ref{schmid-lem} (i) and (ii), 
\begin{align*}
\frac{\wt f_2}{f_1}
&=\frac{\lambda_1}{\lambda_0} + \frac{\lambda_2}{\lambda_0} \frac{\theta F_2}{\theta F_1} - \frac{1}{2 \pi i} \log(1-t)
\\&=\frac{\lambda_1}{\lambda_0} - \frac{1}{2 \pi i} \left(\frac{h_1}{\theta F_1} +\left(\frac{h_2}{\theta F_1}+1\right) \log(1-t)\right)
\end{align*}
is holomorphic at $t=1$. Since $\theta F_1$ is meromorphic at $t=1$, we have $h_2/\theta F_1=-1$. 
Therefore, 
$$\lim_{t \to 1} \frac{\wt f_2}{f_1} = \frac{\lambda_1}{\lambda_0} -\frac{1}{2\pi i} \lim_{t \to 1} \frac{h_1}{\theta F_1}.$$
Since $\theta=q_0+q_1 d/dt$ with $q_i \in \ol\Q(t)$, $\theta F_1$ has Laurent coefficients in $\ol\Q$. 
By \eqref{key3-eq1}, the first Laurent coefficient of $h_1$ is in $k_0 \ol\Q+ k_1 \ol\Q$. 
By a theorem of Gauss (cf. \cite{Bateman} 1.7.3, p.18--19), 
$k_n$ is a linear combination of $\ol\Q$ and $\log(\ol{\Q}^\times)$.
Therefore, we obtain the second assertion from Lemma \ref{key2} and Lemma \ref{schmid-lem} (iii). 
\end{proof}

\section{Proof of the Period Formula}\label{main-pf-sect}
%We prove Theorem \ref{main}.
Let the notation be as in \S \ref{settei-sect}.
Put $D_\infty^{(l)}:=(f^{(l)})^{-1}(\infty)$ and
$\ol{U}^{(l)}:=X^{(l)}\setminus D_\infty^{(l)}=(f^{(l)})^{-1}(\P^1\setminus \{\infty\})$:
\[
\xymatrix{
\ol{U}^{(l)}%\ar@{}[rd]|{\square}
\ar[r]\ar[d]_{f^{(l)}}&X^{(l)}\ar[d]^{f^{(l)}}&D_\infty^{(l)}\ar[d]\ar[l]\\
\P^1\setminus\{\infty\}\ar[r]&\P^1&\{\infty\}\ar[l]
}
\]
Let $K[G^{(l)}]\to K_i$ be a projection and 
$e_i\in K[G^{(l)}]$ the corresponding idempotent
which satisfy
the assumption in Theorem \ref{main}.
We fix an embedding $\chi:K\hra\ol{\Q}$ and an integer $0<k<l$ prime to $l$ 
such that
$\varepsilon_k\ot \chi\in I_i$, i.e. the homomorphism $\varepsilon_k\ot \chi:
K[G^{(l)}]\to\ol{\Q}$
factors through $K_i$ (see \eqref{index} for the definition of $I_i$).
We then write
$\alpha_j=\alpha_j^\chi$ and
$\beta_j=\beta_j^\chi$ simply.
\subsection{}\label{coh-sect}
Put $H^2(X^{(l)})_0:=\Ker[H^2(X^{(l)})\to H^2(X^{(l)}_t)]$ where $X^{(l)}_t$
is the general fiber, and $H^2(U^{(l)})_0$ and $H^2(\ol{U}^{(l)})_0$ similarly.
Recall the commutative diagram \eqref{local1-com}. It induces 
\begin{equation}\label{local1-com-1}
\xymatrix{
0\ar[r]& H^{(l)}\ar[r]\ar[d]^\cong& M^{(l)}\ar[r]\ar[d]^\cong& E\ar[r]\ar@{=}[d]&0\\
0\ar[r]& H^2(X^{(l)})_0/\langle D^{(l)}\rangle\ar[r]& H^2(U^{(l)})_0\ar[r]
& E\ar[r]&0\\
0\ar[r]& H^2(X^{(l)})_0/\langle D^{(l)}\rangle\ar[r]\ar@{=}[u]& 
H^2(\ol{U}^{(l)})_0/\langle D^{(l)}\rangle\ar[r]\ar[u]_\cup
& E_\infty\ar[r]\ar[u]_\cup&0.
}
\end{equation}
Note that all terms are equipped with multiplication by $K[G^{(l)}]$.
Put
\[
H^2(\ol{U}^{(l)})_\fib:=\Ker\left[H^2(\ol{U}^{(l)})\to \prod_F H^2(F)\right]
\]
where $F$ runs over all fibral divisors in $\ol{U}^{(l)}$ (i.e. $f^{(l)}(F)$ is a point). 
Then we claim that
\begin{equation}\label{local1-com-2}
H^2(\ol{U}^{(l)})_\fib\lra
H^2(\ol{U}^{(l)})_0/\langle D^{(l)}\rangle
\end{equation}
is injective.
Indeed, the kernel of the composition
\[
\langle D^{(l)}\rangle\lra H^2(X^{(l)})\lra \prod_F H^2(F)
\]
is one-dimesional, generated by the cycle class of general fiber
(Zariski's lemma, \cite{barth} I, (2.10)), and 
it dies in $H^2(\ol{U}^{(l)})$ as $\ol{U}^{(l)}\ne X^{(l)}$. Since $e_iE_\infty=0$, we have
an injective map
\begin{equation}\label{local1-com-3}
e_iH^2(\ol{U}^{(l)})_\fib\hra
e_iH^2(X^{(l)})_0/\langle D^{(l)}\rangle\cong e_iH^{(l)}.
\end{equation}
\subsection{Rational 2-forms $\om_k$}
Recall from Lemma \ref{key-lem} the 1-form $\omega\in 
\vg(S,F^1\cH^\chi)\subset \vg(S,f_*\Omega^1_{U/S})$ 
and the differential operator $\theta=q_0(t)+q_1(t)d/dt$.
Let $\omega^{(l)}$ be the pull-back of $\omega$ by $t\to t^l$.
Take $q(t)\ne0\in\ol{\Q}[t]$ such that
$p_0(t):=q(t)q_0(t)$ and $p_1(t):=q(t)q_1(t)$ are polynomials
and 
$(t-1) \mid p_1(t)$.
We then consider the rational 2-form
\begin{equation}\label{pff1-1d}
\om_k=\om_{k,q(t)}:=t^{k-1}q(t^l)dt\wedge \omega^{(l)}%\in \vg(U^{(l)},\Omega^2_{X^{(l)}}).
\end{equation}
Replacing $q(t)$ with $t^{a}(1-t)^bq(t)$ for some $a,b\gg0$ if necessary,  
we may assume that
\begin{equation*}\label{pff1-1}
\om_k\in \vg(\ol{U}^{(l)},\Omega^2_{X^{(l)}}).
\end{equation*}
This defines a de Rham
cohomology class $[t^j\om_k]\in H^2_\dR(\ol{U}^{(l)}/\ol{\Q})$ for any $j\geq 0$.
Obviously $[t^j\om_k]\in H^2_\dR(\ol{U}^{(l)}/\ol{\Q})_\fib$
and $\tau [t^j\om_k]=\zeta_l^{k+j}[t^j\om_k]$.
In particular 
\begin{equation}\label{local1-com-4}
[t^{lm}\om_k]\in [e_iH^2_\dR(\ol{U}^{(l)}/\ol{\Q})_\fib]^{\varepsilon_k\ot \chi}
\hra [e_iH^{(l)}_\dR]^{\varepsilon_k\ot \chi}\cong\ol\Q
\end{equation}
for any $m\geq0$ where the inclusion comes from \eqref{local1-com-3}.
\subsection{Homology cycle $\Delta$}\label{delta-sect}
Put
$D^{(l)}_0:=(f^{(l)})^{-1}(0)$.
Let $\delta_{t}\in H_1(X_t,\Q)^\chi$ be the vanishing cycle
in Lemma \ref{lem-basis}.  
By sweeping $\delta_t$
over the segment $0\leq t \leq 1$ we obtain
a Lefschetz thimble $\Delta\in H_2(\ol{U}^{(l)},D^{(l)}_0;\ol{\Q})$.
It may have a nonzero boundary $\partial\Delta\in H_1(D^{(l)}_0,\ol{\Q})$.
Let $\Delta_0^*\subset \P^1$ be the punctured neighborhood of $t=0$,
and put $\wt{\Delta}^*_0:=\pi^{-1}(\Delta^*_0)$.
By the local invariant cycle theorem, there is a canonical isomorphism
\[
H^1(D^{(l)}_0,\ol{\Q})\cong \vg(\wt{\Delta}^*_0,R^1f^{(l)}_*\Q)=\vg(\Delta^*_0,\cM^{(l)}).
\]
Since $k/l+\alpha^\chi_j\not\in\Z$ by the assumption in Theorem \ref{main},
this yields $H_1(D^{(l)}_0,\ol{\Q})^{\varepsilon_k\ot\chi}=0$.
By the exact sequence
\[
H_2(\ol{U}^{(l)},\ol{\Q})^{\varepsilon_k\ot\chi}\lra 
H_2(\ol{U}^{(l)},D^{(l)}_0;\ol{\Q})^{\varepsilon_k\ot\chi}\lra
H_1(D^{(l)}_0,\ol{\Q})^{\varepsilon_k\ot\chi}=0
\]
the component $\Delta^{\varepsilon_k\ot\chi}$ lifts up to a homology cycle
in $H_2(\ol{U}^{(l)},\ol{\Q})^{\varepsilon_k\ot\chi}$, which we write by the same notation
$\Delta^{\varepsilon_k\ot\chi}$.

\subsection{Computing the period of $H^{(l)}$}\label{computing-sect}
We shall compute the period of $[t^{lm}\om_k]$.
We first note that,
since $[t^{lm}\om_k]$ belongs to the ${\varepsilon_k\ot\chi}$-part,
\begin{equation}\label{pff1-lem0}
\int_{\Delta}t^{lm}\om_k=
\int_{\Delta^{\varepsilon_k\ot\chi}}t^{lm}\om_k.
\end{equation}
If one shows the non-vanishing of the integral
for some $m\geq 0$, then we have the non-vanishing $[t^{lm}\om_k]\ne 0$ 
of the cohomology class, and hence it gives a basis of
the $\varepsilon_k\ot\chi$-part of $H^{(l)}_\dR$.
Then the period is given by
\begin{equation}\label{pff1-lem1}
\mathrm{Period}([H^{(l)}]^{\varepsilon_k\ot\chi})
\sim_{\ol{\Q}^\times}\int_{\Delta}t^{lm}\om_k.
\end{equation}
Let us compute the integral \eqref{pff1-lem0}.
\begin{align*}
\int_{\Delta}t^{lm}\om_k
&=
\int_{\Delta}t^{k+lm-1}q(t^l)\,dt\wedge \omega^{(l)}\\
&=
\int_0^1t^{k+lm-1}q(t^l)\,dt\int_{\delta_t}\omega^{(l)}\\
&=\frac{1}{l}
\int_0^1t^{m+k/l-1}q(t)f_1(t)\,dt\\
&=\frac{\lambda_0}{l}\int_0^1 (p_0(t)F_1(t)+p_1(t)F'_1(t))t^{k/l+m-1}dt
\quad\mbox{(Lemma \ref{key-lem})}\\
&=:\frac{\lambda_0}{l}I_m.%\label{pff1-3}
\end{align*}
%\begin{equation}\label{ig-main}
%\exists~m\geq0,\quad
%I_m
%\sim_{\ol{\Q}^\times}
%\frac{\Gamma(q+\alpha_1)\Gamma(q+\alpha_2)}
%{\Gamma(q-\beta_1)\Gamma(q-\beta_2)},
%\quad q:=\frac{k}{l}.
%\end{equation}

%
%
%\begin{prop}\label{pff1-lem2}
%Let $p_0(t)=\sum_id_it^i$ and $p_1(t)=\sum_jd'_jt^j$.
%Put $a:=\alpha_1+\beta_1$ and $b:=\alpha_1+\beta_2$.
%Fix an integer $e\gg 0$ such that
%$c:=\alpha_1+e+k/l>\max(0,a+b-1,a-1,b-1)$.
%Define rational numbers $\kappa_n\in\Q_{>0}$ inductively by
%\begin{equation}\label{pff1-4}
%\kappa_0=1,\quad \kappa_{n+1}/\kappa_n=
%\frac{(c+n)(c+n+1-a-b)}{(c+n+1-a)(c+n+1-b)},\quad
%n=0,1,\cdots.
%\end{equation}
%Then for $m>e$
%\begin{equation}\label{pff1-5}
%I_m=\frac{\Gamma(c)\Gamma(c+1-a-b)}{\Gamma(c+1-a)\Gamma(c+1-b)}
%\left[\sum_i(d_{i-1}+(1-c)d'_{i})\kappa_{m'+i}-d'_{i}(m'+i)\kappa_{m'+i}\right]
%\end{equation}
%where $m\rq{}:=m-e-1$.
%\end{prop}
%\begin{prop}\label{pff1-5-1}
%The term ``$[\cdots]$'' in \eqref{pff1-5}
%does not vanish for infinitely many $m'\geq 0$.
%In particular $I_m\ne0$ and hence $[t^{lm}\om_k]$ is a 
%$\ol{\Q}$-basis of $H_\CM^\chi$.
%\end{prop}
%Propositions \ref{pff1-lem2} and \ref{pff1-5-1} imply
%\[
%I_m
%\sim_{\ol{\Q}^\times}
%\frac{\Gamma(c)\Gamma(c+1-a-b)}{\Gamma(c+1-a)\Gamma(c+1-b)}
%\sim_{\ol{\Q}^\times}
%\frac{\Gamma(q+\alpha_1)\Gamma(q+\alpha_2)}
%{\Gamma(q-\beta_1)\Gamma(q-\beta_2)},\quad q:=\frac{k}{l}\]
%for infinitely many $m\gg 0$.
%This is the desired assertion \eqref{ig-main}, finishing the proof of
%Thm.\ref{main} (2).

%%%
\begin{prop}\label{pff1-5-1}
Let $p_0(t)=\sum_i d_i t^i$and $p_1(t)=\sum_i d_i' t^i$.
Put $q:=k/l$ and
\begin{align}\label{pff1-5-1-a} 
a_n&:=\frac{(\alpha_1+q)_n(\alpha_2+q)_n}{(1-\beta_1+q)_n(1-\beta_2+q)_n},\quad n\geq0,
\\
\label{pff1-5-1-C} 
C_m&:=\sum_{i\ge -1} (d_i- d'_{i+1}(q+m+i)) a_{m+i},\quad m\geq 1
\end{align}
where $d_{-1}:=0$.
Then we have
\[
I_m =C_m \cdot
\Gamma\left({\alpha_1+q,\alpha_2+q \atop 1-\beta_1+q,1-\beta_2+q}\right).
\]
Moreover, for infinitely many $m \ge 1$, we have $C_m \ne 0$ and hence
$[t^{lm}\Omega_k]\ne0$. 
\end{prop}

\begin{proof}
Firstly, recall that
\begin{equation}\label{integral F}
\int_0^1 {}_2F_1\left({a,b\atop d}; xt\right) t^{c-1}(1-t)^{e-c-1}\,dt
=\Gamma\left({c,e-c \atop e}\right) {}_3F_2\left({a,b,c \atop d,e};x\right).
\end{equation}
Using this, we have
\begin{align*}
&\int_0^1F_1(t) t^{q+n-1}\, dt
\\&=\int_0^1 {}_2F_1\left({\alpha_1+\beta_1,\alpha_1+\beta_2\atop 1};t\right)(1-t)^{\alpha_1+q+n-1}\,dt
\\&= \Gamma\left({\alpha_1+q+n\atop \alpha_1+q+n+1}\right) {}_3F_2\left({\alpha_1+\beta_1,\alpha_1+\beta_2,1 \atop 1,\alpha_1+q+n+1};1\right)
\\&= \Gamma\left({\alpha_1+q+n\atop \alpha_1+q+n+1}\right) {}_2F_1\left({\alpha_1+\beta_1,\alpha_1+\beta_2 \atop \alpha_1+q+n+1};1\right)
\\&=\Gamma\left({\alpha_1+q+n, \alpha_2+q+n\atop1-\beta_1+q+n, 1-\beta_2+q+n}\right)
\\&=\Gamma\left({\alpha_1+q, \alpha_2+q\atop1-\beta_1+q, 1-\beta_2+q}\right)a_n
\end{align*}
where we used Euler's formula and $\alpha_1+\alpha_2+\beta_1+\beta_2=1$. 
Hence we have
$$\int_0^1 p_0(t) F_1(t) t^{q+m-1}\, dt = \Gamma\left({\alpha_1+q, \alpha_2+q\atop1-\beta_1+q, 1-\beta_2+q}\right)\sum_i d_i a_{i+m}. $$
Secondly, recall that 
\begin{equation*}
\frac{d}{dt}{}_2F_1\left({a,b \atop c}; t\right) = \frac{ab}{c} {}_2F_1\left({a+1,b+1\atop c+1}; t \right).
\end{equation*}
Using this, we have
\begin{align*}
\frac{d}{dt}F_1(t)
=& \alpha_1 F_1(t)t^{-1} 
\\& -(\alpha_1+\beta_1)(\alpha_1+\beta_2){}_2F_1\left({\alpha_1+\beta_1+1, \alpha_1+\beta_2+1 \atop 2};1-t\right)t^{\alpha_1}. 
\end{align*}
The integral for the first term is already computed. For the second term, we have similarly as above
\begin{align*}
&\int_0^1{}_2F_1\left({\alpha_1+\beta_1+1, \alpha_1+\beta_2+1 \atop 2};1-t\right)t^{\alpha_1+q+n-1}
\\&=\frac{1}{\alpha_1+q+n}
{}_3F_2\left({\alpha_1+\beta_1+1,\alpha_1+\beta_2+1,1\atop \alpha_1+q+n+1,2};1\right). 
\end{align*}
Applying Lemma \ref{otsubo2} with $q=0$, we have
\begin{align*}
&\frac{(\alpha_1+\beta_1)(\alpha_1+\beta_2)}{\alpha_1+q+n} 
{}_3F_2\left({\alpha_1+\beta_1+1,\alpha_1+\beta_2+1,1\atop \alpha_1+q+n+1,2};1\right)\\
&=\frac{1}{\alpha_1+q+n}\Gamma\left({\alpha_1+q+n+1, \alpha_2+q+n-1\atop -\beta_1+q+n, \beta_2+q+n}\right)-1
\\&=\Gamma\left({\alpha_1+q, \alpha_2+q\atop1-\beta_1+q, 1-\beta_2+q}\right)(\alpha_1+q+n-1)a_{n-1}-1.
\end{align*}
Combining these and using $p_1(1)=\sum_i d_i'=0$, we obtain
\begin{align*}
&\int_0^1 p_1(t) F_1'(t)t^{q+m-1}\,dt
\\& =\Gamma\left({\alpha_1+q, \alpha_2+q\atop1-\beta_1+q, 1-\beta_2+q}\right)\sum_i d_i' (\alpha_1 a_{i+m-1} - (\alpha_1+q+i+m-1)a_{i+m-1})
\\& =-\Gamma\left({\alpha_1+q, \alpha_2+q\atop1-\beta_1+q, 1-\beta_2+q}\right)\sum_i d_i' (q+i+m-1)a_{i+m-1}.
\end{align*}
Hence we obtain the first assertion. 
The second assertion follows from the lemma below. 
\end{proof}

\begin{lem}Let $\{a_n\}_{n \ge 0}$ be as above. 
Let $\{x_n\}_{n=1}^r$, $\{y_n\}_{n=1}^r$ be sequences of finite length such that $x_n \ne 0$ for some $n$ and $y_n\ne 0$ for some $n$. 
Then. 
$$\sum_{n=1}^r (x_n a_{n+m} +y_n (n+m)a_{n+m})$$
is non-trivial for infinitely many $m \ge 0$. 
\end{lem}

\begin{proof}
Put
$$e_i=(a_{i+1},\dots, a_{i+r}, (i+1)a_{i+1}, \dots, (i+r)a_{i+r}) \in \Q^{\oplus 2r}.$$
It suffices to show that $e_{m+1},\dots, e_{m+2r}$ are linearly independent. 
Put $a=\alpha_1+q$, $b=\alpha_2+q$, $c=1-\beta_1+q$, $d=1-\beta_2+q$. 
Since $(\alpha)_{i+j}=(\alpha)_i(\alpha+i)_j$, we have for $j=1,\dots, r$, 
\begin{align*}
a_{i+j}
&=\frac{(a)_i(b)_i}{(c)_i(d)_i} \cdot \frac{(a+i)_j(b+i)_j}{(c+i)_j(d+i)_j}
\\&=\frac{(a)_i(b)_i}{(c)_i(d)_i} \cdot \frac{(a+i)_j(b+i)_j(c+i+j)_{r-j}(d+i+j)_{r-j}}{(c+i)_{r}(d+i)_{r}}. 
\end{align*}
So we have the determinant
\begin{align*}
\begin{vmatrix}e_{m+1}\\ \vdots \\ e_{m+2r} \end{vmatrix}
= \prod_{i=m+1}^{m+2r} \frac{(a)_i(b)_i}{(c)_i(d)_i(c+i)_{r}(d+i)_{r}}\cdot 
\begin{vmatrix} f_{m+1} \\ \vdots \\ f_{m+2r}
\end{vmatrix}
\end{align*}
where we put
$$f_i=(b_{i,1},\dots, b_{i,r}, (i+1)b_{i,1},\dots, (i+r)b_{i,r})$$
with
$$b_{i,j}= (a+i)_j(b+i)_j(c+i+j)_{r-j}(d+i+j)_{r-j}.$$
For each $j=1,\dots, r$, 
$$P_j(t):=(a+t)_j(b+t)_j(c+t+j)_{r-j}(d+t+j)_{r-j}$$
is a polynomial of degree $2r$ such that $P_j(i)=b_{i,j}$. 
Suppose that the above determinant is $0$. 
Then, there exist $c_j$, $d_j$ which are not all $0$ such that the polynomial $\sum_{j=1}^r (c_j+d_j t)P_j(t)$ vanishes at $t=m+1,\dots, m+2r$. 
Since every $P_j(t)$ is divisible by $(a+t)(b+t)$ which does not vanish at integers, we have a polynomial of degree $2r-1$ with $2r$ distinct roots. Hence we have $\sum_{j=1}^r (c_j+d_j t)P_j(t)=0$.  
Since $P_1(t), \dots , P_{r-1}(t)$ are divisible by $(c+t+r-1)(d+t+r-1)$, so is $(c_r+d_r t)P_r(t)$. 
On the other hand, by Lemma \ref{lem-basis}, $P_r(t)$ is not divisible by $(c+t+r-1)(d+t+r-1)$, 
hence $c_r=d_r=0$. Proceeding similarly, we obtain $c_j=d_j=0$ for all $j$, which is a contradiction. 
\end{proof}
We finish the proof of Theorem \ref{main}.
The constant $C_m$ is an nonzero algebraic number for infinitely many $m$'s.
Then \eqref{pff1-lem1} holds so that we have
\[\mathrm{Period}([H^{(l)}]^{\varepsilon_k\ot\chi}) 
\sim_{\ol\Q^\times} 
\lambda_0 I_m
\sim_{\ol\Q^\times} 
\lambda_0\cdot
 \Gamma\left({\alpha_1+q,\alpha_2+q \atop 1-\beta_1+q,1-\beta_2+q}\right).
 \]
Since $\lambda_0\in 2\pi i\ol{\Q}^\times$ by Lemma \ref{key2}, 
we are done.

\begin{rem}Since $\alpha_1+\alpha_2+\beta_1+\beta_2\in\Z$, we can also write
$$\mathrm{Period}([H^{(l)}]^{\varepsilon_k\ot\chi}) 
\sim_{\ol\Q^\times} B(\alpha_1+q, \beta_1-q)B(\alpha_2+q,\beta_2-q).$$
\end{rem}
The above proof also shows that the map \eqref{local1-com-4} is bijective.
Hence
\begin{cor}\label{local1-com-5}
$e_iH^2(\ol{U}^{(l)})_\fib\cong e_i H^{(l)}$.
\end{cor}

\section{Proof of the Regulator Formula}\label{main-reg-pf-sect}
In this section we prove Theorem \ref{main-reg}.
We fix a projection $K[G^{(l)}]\to K_i$ and the idempotent $e_i$ which satisfy
the assumption of Theorem \ref{main}, i.e. none of \eqref{main-cond} is an integer.
Recall the following notations:  
\begin{align*}
&
D^{(l)}_0:=(f^{(l)})^{-1}(0),\quad
D^{(l)}_\infty:=(f^{(l)})^{-1}(\infty),\quad 
D^{(l)}_i:=(f^{(l)})^{-1}(\zeta^i_l),~(1\leq i\leq l)
\\&
D_{ss}^{(l)}:=\sum_{i=1}^lD^{(l)}_i,\quad
D^{(l)}:=D^{(l)}_0+D^{(l)}_\infty+D^{(l)}_{ss},\quad  \E:=D^{(l)}_0+D^{(l)}_{ss}, 
\\&
\ol{U}^{(l)}:=X^{(l)}\setminus D^{(l)}_\infty, \quad  U^{(l)}:=X^{(l)}\setminus D^{(l)}. 
\end{align*}

\subsection{Cycle $\Gamma$}\label{gamma-sect}
Let $\gamma_{\Q,t}\in H_1(X_t,\Q)$ be a homology cycle
which does not vanish at $t=1$.
We then define a cycle \[\Gamma\in H_2(\ol{U}^{(l)},D_0^{(l)}+D^{(l)}_l;\Q)\] 
to be the Lefschetz thimble
obtained by sweeping $\gamma_{\Q,t}$ over the segment $0\leq t \leq1$.
Since $k/l+\alpha^\chi_j\not\in\Z$ by the assumption of Theorem \ref{main},
one has $e_iH_1(D_0,\ol\Q)=0$
(cf. \S \ref{delta-sect}).
Hence we obtain a cycle
\[e_i\Gamma
\in H_2(\ol{U}^{(l)},D_{ss}^{(l)};\Q),
%\quad\Gamma^{\varepsilon_k\ot\chi}\in H_2(\ol{U}^{(l)},D_{ss}^{(l)};\ol\Q),
\quad D^{(l)}_{ss}:=
\sum_{i=1}^lD^{(l)}_i\]
with nontrivial boundary:
\[
\partial(e_i\Gamma)=e_i\gamma_{\Q,1}\ne0
\in e_i H_1(D^{(l)}_{ss},\Q)\cong K_i.
\]

\subsection{Proof of Theorem \ref{main-reg} : Step 1}
We want to compute the extension data of \eqref{local1}.
The auto-duality \eqref{auto-dual} for $H^{(l)}$ together with the Verdier duality
yields a commutative diagram
\[
\xymatrix{
0\ar[r] &H^{(l)}(2)\ar[r]\ar[d]^\cong& 
M^{(l)}(2)\ar[r]\ar[d]^\cong& E(2)\ar[r]\ar[d]^\cong&0\\
0\ar[r] &(H^{(l)})^*\ar[r]\ar[d]& 
(H^1(\P^1,j_!M^{(l)}))^*\ar[r]\ar[d]& (i^{-1}j_*\cM^{(l)})^*\ar[r]\ar[d]^\cong&0\\
0\ar[r] &H_2(X^{(l)})/H_{2}(D^{(l)})\ar[r]& 
H_2(X^{(l)},D^{(l)})\ar[r]& H_1(D^{(l)})\ar[r]&0\\
0\ar[r] &H_2(\ol{U}^{(l)})/H_2(\E)\ar[r]\ar[u]& 
H_{2}(\ol{U}^{(l)},\E)\ar[r]\ar[u]& 
H_1(\E)\ar[u]
}
\]
with exact rows where $j:\P^1\setminus\{0,1,\infty\}\hra\P^1$ and $i:\{0,1,\infty\}\hra\P^1$.
It follows from Corollary \ref{local1-com-5} that the exact sequence
\begin{equation}\label{reg-1}
\xymatrix{
0\ar[r] &e_iH_2(\ol{U}^{(l)})/H_2(\E)\ar[r]& 
e_iH_2(\ol{U}^{(l)},\E)\ar[r]& 
e_iH_1(\E)\ar[r]\ar@{=}[d]&0\\
&[e_iH^2(\ol{U}^{(l)},\Q)_\fib]^*\ar[u]^\cong&&
e_iH_1(D^{(l)}_{ss})
}
\end{equation}
is isomorphic to the $e_i$-part of \eqref{local1}.

\medskip

Let us discuss \eqref{reg-1}.
Since $e_iH^1(D_{ss},\Q)$ is a Hodge structure of type $(0,0)$, 
there is an isomorphism
\[
\Pi:e_i F^1H^2_\dR(\ol{U}^{(l)},\E)
\os{\cong}{\lra} e_i F^1H^2_\dR(\ol{U}^{(l)})_\fib.
\]
For $\omega\in e_i F^1H^2_\dR(\ol{U}^{(l)})_\fib$ we write
\[\omega_{(\ol{U}^{(l)},\E)}:=\Pi^{-1}(\omega)\in e_i F^1H^2_\dR(\ol{U}^{(l)},\E).\]
The following is well-known to specialists and 
indeed it can be proven immediately from the definition (the detail
 is left to the reader).
\begin{prop}
Let 
\[
\check{\rho}:e_iH_1(D_{ss}^{(l)},\Q)\lra\Ext^1_\MdRH(\Q,e_iH^2(\ol{U}^{(l)})_\fib)
\]
be the connecting homomorphism
arising from \eqref{reg-1}.
Let $\Gamma_x\in e_iH_2(\ol{U}^{(l)},\E;\Q)$ be a lifting of $x\in 
e_iH_1(D_{ss}^{(l)},\Q)$.
Then, under the canonical isomorphism 
\[
\Ext^1_\MdRH(\Q,e_iH^2(\ol{U}^{(l)})_\fib)\cong
\Coker\left[e_iH_2(\ol{U}^{(l)},\Q)
\to \Hom(e_iF^1H^2_\dR(\ol{U}^{(l)})_\fib,\C)\right]
\]
we have
\[
\check{\rho}(x)=\left[
\omega\longmapsto\int_{\Gamma_x}
\omega_{(\ol{U}^{(l)},\E)}\right].
\]
\end{prop}

Let us see $\check{\rho}(x)$ more explicitly.
We first note that one can choose a lifting $\Gamma_x$ by applying 
an element $\alpha_x\in K[G^{(l)}]$ on the cycle 
$\Gamma$ constructed in \S \ref{gamma-sect}:
\[
\Gamma_x=\alpha_x\Gamma.
\]
Let $\omega$ be a de Rham cohomology class defined over $\ol{\Q}$.
Then $\omega_{(\ol{U}^{(l)},\E)}$ is a class defined over $\ol{\Q}$ as well. 
Let $\wt{\omega}\in \vg(\ol{U}^{(l)},\Omega^2_{\ol{U}^{(l)}/\ol{\Q}})$ satisfy
$[\wt{\omega}]=\omega$ in $H^2_\dR(\ol{U}^{(l)}/\ol{\Q})$.
Then since $\partial\Gamma_x\in H_1(D_{ss},\Q)$ we have
\begin{equation}\label{reg-exact1}
\int_{\Gamma_x}
\omega_{(\ol{U}^{(l)},\E)}=
\int_{\Gamma_x}
\wt{\omega}+c,\quad
\exists c\in\Image(H^1_\dR(D_{ss}/\ol{\Q})\ot H_1(D_{ss},\Q))=\ol{\Q}
\end{equation}
(cf. \cite{formula} \S 3.3).
Recall the rational 2-forms $t^{lm}\om_k$ from
the proof of period formula in \S \ref{main-pf-sect}.
It gives a basis of the $\varepsilon_k\ot\chi$-part of $e_i H^2_\dR(\ol{U}^{(l)})_\fib$.
Now let $\wt{\omega}=t^{lm}\om_k$. 
We then have
\begin{equation*}\label{reg-exact2}
\int_{\Gamma_x}
\omega_{(\ol{U}^{(l)},\E)}=
\int_{g\Gamma}
t^{lm}\om_k+c
=(\varepsilon_k\ot\chi)(\alpha_x)\cdot\int_\Gamma t^{lm}\om_k+c.
\end{equation*}
Finally let us recall the connecting homomorphism $\rho$ \eqref{connecting-def}
which arises from \eqref{local1}.
Since $\varepsilon_{-k}\ot{}^t\chi$-part
of \eqref{local1} corresponds to $\varepsilon_{k}\ot\chi$-part
of \eqref{reg-1}, we have
\[
\rho^{\varepsilon_{-k}\ot{}^t\chi}=
(\check\rho)^{\varepsilon_k\ot\chi}.
\]
Summing these up, we obtain the following.
\begin{prop}\label{prop-reg-exact3}
Let $\varepsilon_k\ot\chi\in I_i$.
Then for $x\in e_iE(2)$
there is a constant $c\in\ol{\Q}$ such that
\begin{equation}\label{reg-exact3}
\rho^{\varepsilon_{-k}\ot{}^t\chi}(x)=(\varepsilon_k\ot\chi)(\alpha_x)
\int_\Gamma t^{lm}\om_k+c.
\end{equation}
\end{prop}
Note $(\varepsilon_k\ot\chi)(\alpha_x)\in\ol{\Q}^\times$ unless $x=0$.

\begin{rem}
The constant ``$c$'' in \eqref{reg-exact3} depends on the choice of 
the lifting $\wt{\omega}=t^{lm}\om_k$ (recall from Proposition \ref{pff1-5-1}
that $t^{lm}\om_k$ can be a basis
for infinitely many $m$'s).
However if one chooses $\wt{\omega}$
to be a certain lifting arising from Deligne\rq{}s canonical extension,
then it is proven that $c=0$ (\cite{asakura-otsubo} Appendix).
\end{rem}

\subsection{Proof of Theorem \ref{main-reg} : Step 2 : Contiguous relations of ${}_3F_2$}
\begin{lem}\label{bailey}
If $c+1>a+b$ and $q>0$, we have
\[
{}_3F_2\left(\begin{matrix}
a,b,q\\
c,q+1
\end{matrix};1\right)=
q\Gamma\left({c,c+1-a-b \atop c+1-a,c+1-b}\right)
{}_3F_2\left(\begin{matrix}
1,c+1-a-b,c-q\\
c+1-a,c+1-b
\end{matrix};1\right).
\]
\end{lem}
\begin{proof}
Apply \cite{Bailey} Ch. III, 3.2 (1), p.14.
\end{proof}
\begin{lem}[3-term relations]\label{cont1}
Let 
\[
F^q(x):={}_3F_2\left(\begin{matrix}
1,c,q\\
a,b
\end{matrix};x\right), \quad
F_a(x):={}_3F_2\left(\begin{matrix}
1,c,q\\
a,b
\end{matrix};x\right).
\]
Then we have
\begin{multline*}
(a-q-1)(b-q-1)F^q(x)+q(a+b-3-2q-(c-q-1)x)F^{q+1}(x)\\
+q(1+q)(1-x)F^{q+2}(x)=(a-1)(b-1),
\end{multline*}
and
\begin{multline*}
(a-2)(a-1)(1-x)F_{a-2}(x)+(a-1)((2a-c-q-3)x-a+b+1)F_{a-1}(x)\\
-(a-q-1)(a-c-1)xF_a(x)=(a-1)(b-1).
\end{multline*}
In particular, if $a+b>c+q+2$, we have
\begin{align*}
&(a-q-1)(b-q-1)F^q(1)+q(a+b-c-2-q)F^{q+1}(1)=(a-1)(b-1),
\\&
(a-1)(a+b-c-q-2)F_{a-1}(1)
-(a-q-1)(a-c-1)F_a(1)=(a-1)(b-1).
\end{align*}
\end{lem}
\begin{proof}
Let $D=x\frac{d}{dx}$ be the Euler differential operator.
Then $F^q(x)$ is a solution of the differential operator
\[
D(D+a-1)(D+b-1)-x(D+1)(D+c)(D+q)
=D[(D+a-1)(D+b-1)-x(D+c)(D+q)].
\]
On the other hand, one can directly shows
$(D+q)F^q(x)=qF^{q+1}(x)$.
Therefore 
if we write
\[
(D+a-1)(D+b-1)-x(D+c)(D+q)=
a_1(x)(D+q+1)(D+q)+a_2(x)(D+q)+a_3(x)
\]
then we have
\[
D(a_1(x)q(q+1)F^{q+2}+a_2(x)qF^{q+1}(x)+a_3(x)F^q(x))=0
\]
$\Longleftrightarrow$
\[
a_1(x)q(q+1)F^{q+2}+a_2(x)qF^{q+1}(x)+a_3(x)F^q(x)=\mbox{constant.}
\]
We thus obtain the 3-term relation for $F^q$ (details are left to the reader).
Noting $(D+a-1)F_a(x)=(a-1)F_{a-1}(x)$, the 3-term relation for $F_a$ is proven in the same way.
\end{proof}
Lemmas \ref{bailey} and \ref{cont1} immediately imply
\begin{cor}\label{cont2}
For $c+1>a+b$, put
\[
F_c^{a,b,q}:=\Gamma\left({c+1-a,c+1-b\atop c,c+1-a-b }\right)
{}_3F_2\left(\begin{matrix}
a,b,q\\
c,q+1
\end{matrix};1\right).
\]
Then for any rationals $a'\equiv a$, $b'\equiv b$, $c'\equiv c$
and $q'\equiv q$ mod $\Z$, there are rationals $k$, $k'$, $k''$ such that
\[
k F_c^{a,b,q}+k'F_{c'}^{a',b',q'}+k''=0.
\]
\end{cor}
We shall apply Corollary \ref{cont2} to the case $a'=a$, $b'=b$, 
$c'=c$ and $q'=q+1$.
For the later use we write it down explicitly.
\begin{lem}\label{otsubo2}
For any $a,b,c \in\R$ with $c+1>a+b$, put
\[
F(q):={}_3F_2\left({a,b,q\atop c,q+1};1\right).\]
Then, we have
\[
\frac{(q+1-a)(q+1-b)}{q+1}F(q+1)-(q+1-c)F(q)=
\Gamma\left({c, c+1-a-b \atop c-a,c-b}\right).
%\frac{\Gamma(c)\Gamma(c+1-a-b)}{\Gamma(c-a)\Gamma(c-b)}.
\]
\end{lem}

\subsection{Proof of Theorem \ref{main-reg} : Step 3}
We finish the proof of Theorem \ref{main-reg}. 
By Proposition \ref{prop-reg-exact3}
it is enough to show that the integral 
\[\int_\Gamma t^{lm}\om_k\]
in \eqref{reg-exact3}
is a $\ol\Q$-linear combination of the terms \eqref{main-reg-term0} and
\eqref{main-regterm}, and that the coefficient of \eqref{main-regterm} is nonzero.
Put as in Lemma \ref{key-lem}
\[
F_1(t)=t^{\alpha_1}{}_2F_1\left({\alpha_1+\beta_1,\alpha_1+\beta_2\atop1};t\right),\quad 
F_2(t)=t^{\alpha_1}{}_2F_1\left({\alpha_1+\beta_1,\alpha_1+\beta_2,\atop 2\alpha_1+\beta_1+\beta_2};t\right). 
\]
Similarly as in the proof in \S \ref{computing-sect}, we have
\begin{align}
\int_\Gamma t^{lm}\om_k
&\os{\eqref{pff1-1d}}{=}
\int_0^1t^{lm+k-1}q(t^l)\,dt \int_{\gamma_t}\omega^{(l)}\notag\\
&=
\frac{1}{l}\int_0^1t^{m+k/l-1}q(t)f_2(t)\, dt\notag\\
&=\frac{1}{l}
\int_0^1t^{m+k/l-1}q(t)(\lambda_1\theta F_1+\lambda_2 \theta F_2) \,dt
\quad (\mbox{Lemma \ref{key-lem}})\notag\\
&=\frac{\lambda_1}{l}I_m+
\frac{\lambda_2}{l}
\int_0^1t^{m+k/l-1}\bigl(p_0(t)F_2(t)+p_1(t) F'_2(t)\bigr)\,dt\label{pff3-eq1}. 
\end{align}
The integral $I_m$ is computed in Proposition \ref{pff1-5-1}.
Let us compute the second integral in \eqref{pff3-eq1}:
\begin{align*}
J_m:&=\int_0^1t^{m+k/l-1}\bigl(p_0(t)F_2(t)+p_1(t) F'_2(t)\bigr)\,dt\\
&=\int_0^1\bigr(t^{m+k/l-1}p_0(t)-(t^{m+k/l-1}p_1(t))' \bigr)F_2(t)\,dt,
\end{align*}
where the second equality follows from 
$(t-1)|p_1(t)$ as is assumed.
Let $p_0(t)=\sum d_it^i$ and $p_1(t)=\sum d'_it^i$.
We fix a sufficiently large integer $m$ such that $C_m\ne0$ which is defined in
Proposition \ref{pff1-5-1}.
Put 
\[
K_n:=\int_0^1t^{k/l+n-1} F_2(t)\,dt
=\int_0^1t^{\a_1+k/l+n-1} {}_2F_1\left({\a_1+\b_1,\a_1+\b_2 \atop
2\a_1+\b_1+\b_2};t\right) \,dt.\]
Then 
$J_m$ is a linear combination of $K_n$'s:
\begin{equation}\label{jm1}
J_m=\sum_{i\geq -1}(d_i-(m+k/l+i)d'_{i+1})K_{m+i},\quad (d_{-1}:=0).
\end{equation}

\begin{lem}\label{otsubo1}
Put $a:=\alpha_1+\beta_1$, $b:=\alpha_1+\beta_2$ and $q=k/l$. 
Then we have
\[
K_n=(q+\a_1+n)^{-1}\cdot {}_3F_2\left(\begin{matrix}a,b,
q+\a_1+n\\a+b,q+\a_1+n+1\end{matrix};1
\right).
\]
\end{lem}
\begin{proof}
Straightforward from \eqref{integral F}.
\end{proof}
Lemmas \ref{otsubo2} and \ref{otsubo1} yield that
there are $p_n,p\rq{}_n\in \Q$
such that
\begin{equation}\label{jm2}
B(a,b)K_n=p_nB(a,b){}_3F_2\left(\begin{matrix}a,b,\alpha_1+q\\
a+b,\alpha_1+q+1\end{matrix};1
\right)+p'_n
\end{equation}
Using $\alpha_2=1-(\alpha_1+\beta_1+\beta_2)$,
the rational number $p_n$ is given by 
\begin{align*}
p_n&=(q+\alpha_1+n)^{-1}\prod_{k=1}^n
\frac{(q+\alpha_1+k)(q+\alpha_2+k-1)}{(q-\beta_1+k)(q-\beta_2+k)}
\\&=(q+\alpha_1)^{-1}
\frac{(q+\alpha_1)_n(q+\alpha_2)_n}{(q-\beta_1+1)_n(q-\beta_2+1)_n}.
\end{align*}
We thus have $p_n=(q+\a_1)^{-1}a_n$ where $a_n$ is the constant in \eqref{pff1-5-1-a}. 
Hence
\begin{align*}
B(a,b)J_m
&=(q+\alpha_1)^{-1}\left(\sum_{i\geq -1}(d_i-(m+k/l+i)d'_{i+1})a_{m+i}\right)
\\&\qquad \cdot  B(a,b){}_3F_2\left(\begin{matrix}a,b,\alpha_1+q\\
a+b,\alpha_1+q+1\end{matrix};1
\right)+p''_m\\
&=C_mB(a,b){}_3F_2\left(\begin{matrix}a,b,\alpha_1+q\\
a+b,\alpha_1+q+1\end{matrix};1
\right)+p''_m
\end{align*}
for some $p''_m\in\ol\Q$ where $C_m$ is the constant in \eqref{pff1-5-1-C}.
Since $\lambda_2\ne0\in \ol{\Q}\cdot B(a,b)$ by Key Lemma 3 (Lemma \ref{key3}), 
we have
\begin{align*}
\int_\Gamma t^{lm}\om_k
&= \frac{\lambda_1}{l}I_m+\frac{\lambda_2}{l}J_m\\
&= \frac{\lambda_1}{l}I_m+\frac{\lambda_2}{l}\left(
C_m\cdot{}_3F_2\left(\begin{matrix}a,b,\alpha_1+q\\
a+b,\alpha_1+q+1\end{matrix};1
\right)+C'_mB(a,b)^{-1}\right)\\
&=
\frac{\lambda_1}{l}I_m+c_1+c_2 
B(a,b)~{}_3F_2\left(\begin{matrix}a,b,\alpha_1+q\\
a+b,\alpha_1+q+1\end{matrix};1
\right)
\end{align*}
for some $c_1\in\ol\Q$ and $c_2\in\ol\Q^\times$.
The third term appears as \eqref{main-regterm}.
Again, by Key Lemma 3 (Lemma \ref{key3}) and Proposition \ref{pff1-5-1},
the first term $\lambda_1I_m/l$ appears as 
the second term in \eqref{main-reg-term0}.
So we are done.

\end{document}